\documentclass{amsart}
\usepackage{amssymb}
\usepackage{amsmath}
\usepackage{amsthm}
\usepackage{mathrsfs}
\usepackage{hyperref}
\usepackage[utf8]{inputenc}

\title[Weak capacity and modulus comparability]{Weak capacity and modulus comparability in Ahlfors regular metric spaces} 
\author{Jeff Lindquist}
\email{JLindquist@math.ucla.edu}
\address{Department of Mathematics, University of California, Los Angeles,  Box 95155, Los Angeles, CA, 90095-1555, USA}
\thanks{The author was partially supported by NSF grants DMS-1506099 and DMS-1162471.  This work is based on the author's forthcoming thesis}

\newtheorem{thm}{Theorem}[section]
\newtheorem{lemma}[thm]{Lemma}
\newtheorem{cor}[thm]{Corollary}
\theoremstyle{definition}
\newtheorem{defn}[thm]{Definition}
\newtheorem*{ack}{Acknowledgements}
\newtheorem{remark}[thm]{Remark}
\newtheorem{prop}[thm]{Proposition}

\newcommand{\N}{\mathbb{N}}

\newcommand{\R}{\mathbb{R}}
\newcommand{\C}{\mathbb{C}}
\newcommand{\mH}{\mathscr{H}}
\newcommand{\Hc}{\mathscr{H}^1_\infty}
\newcommand{\Hq}{\mathscr{H}^Q_\infty}
\newcommand{\norm}[1]{\left\lVert#1\right\rVert}
\DeclareMathOperator{\diam}{diam}
\DeclareMathOperator{\ARCdim}{ARCdim}
\DeclareMathOperator{\dist}{dist}
\DeclareMathOperator{\qcap}{wcap_\mathnormal{Q}}
\DeclareMathOperator{\capq}{wcap_\mathnormal{q}}
\DeclareMathOperator{\pcap}{wcap_\mathnormal{p}}
\DeclareMathOperator{\wcap}{wcap}
\DeclareMathOperator{\qmod}{mod_\mathnormal{Q}}
\DeclareMathOperator{\modp}{mod_\mathnormal{p}}
\DeclareMathOperator{\wcqcap}{wc-cap_\mathnormal{Q}}
\DeclareMathOperator{\wcpcap}{wc-cap_\mathnormal{p}}
\DeclareMathOperator{\dimh}{dim_{\mH}}

\def\Xint#1{\mathchoice
{\XXint\displaystyle\textstyle{#1}}%
{\XXint\textstyle\scriptstyle{#1}}%
{\XXint\scriptstyle\scriptscriptstyle{#1}}%
{\XXint\scriptscriptstyle\scriptscriptstyle{#1}}%
\!\int}
\def\XXint#1#2#3{{\setbox0=\hbox{$#1{#2#3}{\int}$ }
\vcenter{\hbox{$#2#3$ }}\kern-.6\wd0}}

\def\dashint{\Xint-}

\begin{document}
\maketitle

% ABSTRACT

\begin{abstract}
Let $(Z,d,\mu)$ be a compact, connected, Ahlfors $Q$-regular metric space with $Q>1$. Using a hyperbolic filling of $Z$, we define the notions of the $p$-capacity between certain subsets of $Z$ and of the weak covering $p$-capacity of path families $\Gamma$ in $Z$.  We show comparability results and quasisymmetric invariance. As an application of our methods we deduce a result due to Tyson on the geometric quasiconformality of quasisymmetric maps between compact, connected Ahlfors $Q$-regular metric spaces.
\end{abstract}

% SECTION 1

\section{Introduction}

Modulus of path families has become an important tool in studying metric spaces with a rich supply of rectifiable paths.  The existence of sufficiently many rectifiable paths, however, is not guaranteed.  For instance, starting from a metric space $(X,d)$, one sees the ``snowflaked'' metric space $(X,d^\alpha)$ with $\alpha \in (0,1)$ carries no nonconstant rectifiable paths.  Accordingly, traditional modulus techniques are insufficient in many cases.

In this paper we will study metric measure spaces $(Z, d, \mu)$ which are compact, connected, and Ahlfors $Q$-regular with $Q>1$.  This means $(Z,d)$ is a separable metric space and $\mu$ is Borel regular.  The last condition is one on the volume growth of balls: specifically, a ball $B$ of radius $r$ has $\mu$-measure comparable to $r^Q$.

We develop two rough extensions of modulus to a ``hyperbolic filling'' associated with a given metric space.  A hyperbolic filling $X = (V,E)$ of $(Z,d,\mu)$ is a graph with vertices that correspond to metric balls and an edge structure which mirrors the combinatorial structure of our metric space.  For a useful picture to have in mind consider the unit disk model of the hyperbolic space $\mathbb{H}^2$.  Here the outer circle $S^1$ plays the role of our metric measure space $Z$ and the hyperbolic filling can be interpreted as a graph representing a Whitney cube decomposition of the interior.  In this setting, cubes correspond to vertices and are connected by edges if they intersect.  Hyperbolic fillings are Gromov hyperbolic metric spaces when endowed with the graph metric.  Moreover, our original space can be identified as the boundary at infinity $\partial_\infty X= Z$ following a standard construction found in \cite[Chapter 2]{BuS}.

Hyperbolic fillings are well defined up to a scaling parameter and a choice of a vertex set at each scale.  The extensions of modulus presented below are essentially quasi-isometrically invariant; changing the given hyperbolic filling will change the quantities by a controlled multiplicative amount.  This multiplicative ambiguity also appears in the modulus comparison results and hence no generality is lost by working with a fixed hyperbolic filling for each metric space.  The general construction of hyperbolic fillings follows \cite{BP} and \cite{BuS} and is detailed in Section \ref{Sec Hyp Fill} along with some of the useful properties of such fillings.

Generalizations of modulus are not new; in \cite{P1} and \cite{P2} Pansu develops a generalized modulus which is adapted in \cite{Ty}.  One key advantage of these generalized notions of modulus, as here, is that proving quasisymmetric invariance is relatively straightforward after setting up the appropriate definitions.

In our definitions we will need the notion of the weak $\ell^p$-norm of a function with a countable domain.  Let $X$ be a countable space and $f \colon X \to \C$.  We define $\norm{f}_{p,\infty}$ as the infimum of all constants $C>0$ such that
\begin{equation*}
\# \{x : |f(x)| > \lambda \} \leq \frac{C^p}{\lambda^p}
\end{equation*}
for all $\lambda > 0$.  We note that in general $\norm{f}_{p,\infty}$ is not a norm but for $p>1$ it is comparable to a norm (see \cite[Section 2]{BoS}).  We freely interchange the two and refer to $\norm{f}_{p,\infty}$ as the weak $\ell^p$-norm of $f$.  The use of the weak norm in the following definitions is motivated by \cite{BoS}.

We now define one of the two quantities used in this paper.  We work with a compact, connected, Ahlfors $Q$-regular metric measure space $(Z,d,\mu)$ with hyperbolic filling $X = (V,E)$.   Both quantities are defined in a similar manner as modulus: certain functions defined on the hyperbolic filling are admissible if they give enough length to an appropriate collection of paths. Then to define the quantity in question we infimize over the $p$-th power of the weak $\ell^p$-norm  of all admissible functions.

The first quantity, weak $p$-capacity ($\pcap$), is defined both for pairs of open sets with $\dist(A,B) > 0$ and for disjoint continua. A continuum is a compact, connected set that consists of more than one point.  We use the notation
\begin{equation*}
\dist(A,B) = \inf\{d(a,b):a\in A, b\in B\}
\end{equation*}
for the distance between $A$ and $B$.  The main idea is that instead of connecting two such sets by paths lying in $Z$, we look at the (necessarily infinite) paths connecting $A$ and $B$ in the hyperbolic filling.  More precisely, given $A,B \subseteq Z$ we call a function $\tau \colon E \to [0,\infty]$ admissible for $A$ and $B$ if for all infinite paths $\gamma \subseteq E$ with nontangential limits in $A$ and $B$ we have $\sum_{e \in \gamma}{\tau(e)}\geq 1$ (see Section \ref{Sec Hyp Fill} for boundaries at infinity of hyperbolic fillings and what it means for a path to have nontangential limits).  If $A$ and $B$ are understood we just call $\tau$ admissible.

\begin{defn}
Given open sets $A, B \subseteq Z$ with $\dist(A,B)>0$ or disjoint continua $A,B \subseteq Z$, we define the weak $p$-capacity $\pcap(A,B)$ between $A$ and $B$ as 
\begin{equation*}
\pcap(A,B) = \inf \{ \norm {\tau}_{p,\infty}^p : \tau \  \text{is admissible for $A$ and $B$} \}.
\end{equation*}
\end{defn}
Proposition \ref{positivity pcap} states that when this is defined for open sets, $\pcap(A,B)>0$.  It is also true that for fixed sets $A$ and $B$ we have $\pcap(A,B) \to 0$ as $p \to \infty$.

We remark again here that in the definition of $\pcap$ there is an implicit choice of a fixed hyperbolic filling and that by changing the hyperbolic filling we may change the value of $\pcap$ by a multiplicative constant.  That is, if $\wcap'_p$ is defined as $\pcap$ with a different hyperbolic filling, then there are constants $c,C > 0$ such that for all open $A$ and $B$ with $\dist(A,B) > 0$ and disjoint continua, one has
\begin{equation*}
c \pcap(A,B) \leq \wcap'_p(A,B) \leq C \pcap(A,B).
\end{equation*}
This follows from the proof of Theorem \ref{QS inv qcap} as hyperbolic fillings of the same metric space are quasi-isometric.  For this reason we ignore the dependence on the hyperbolic filling in the statements of the theorems below.

Our first main result shows that in general $\qcap$ is larger than the $Q$-modulus of the path family connecting $A$ and $B$ (denoted $\qmod(A,B)$; for the definition of $\qmod$ of a path family, see Section \ref{Sec Preliminaries}).

\begin{thm}
\label{qmod < qcap}
Let $Q>1$ and let $(Z,d,\mu)$ be a compact, connected Ahlfors $Q$-regular metric space.  Then there exists a constant $C>0$ depending only on $Q$ and the hyperbolic filling parameters with the following property: whenever $A,B \subseteq Z$ are either open sets with $\dist(A,B) > 0$ or disjoint continua,
\begin{equation*}
\qmod(A,B)  \leq C\qcap(A,B).
\end{equation*}
\end{thm}

For a Loewner space (see Section \ref{Sec Preliminaries} or \cite[Chapter 8]{He} for the precise definition), $\qcap$ is comparable to this modulus.

\begin{thm}
\label{qcap < qmod}
Let $Q>1$ and let $(Z,d,\mu)$ be a compact, connected Ahlfors Q-regular metric space which is also a Q-Loewner space.  Then there exist constants $c,C>0$ depending only on $Q$ and the hyperbolic filling parameters with the following property: whenever $A,B \subseteq Z$ are either open sets with $\dist(A,B) > 0$ or disjoint continua,
\begin{equation*}
c \qmod(A,B) \leq \qcap(A,B) \leq C \qmod(A,B).
\end{equation*}
\end{thm}

Hence $\qcap$ is a quantity that agrees with $\qmod$ up to a multiplicative constant, at least for path families connecting appropriate sets, on spaces with a large supply of rectifiable paths.  We also prove $\pcap$ satisfies a quasisymmetric invariance property.  Given a homeomorphism $\eta: [0,\infty) \to [0,\infty)$, a map $\varphi \colon Z \to W$ is called an $\eta$-quasisymmetry if whenever $z, z', z'' \in Z$ satisfy $|z-z'| \leq t |z - z''|$, we have $|\varphi(z) - \varphi(z')| \leq \eta(t)|\varphi(z) - \varphi(z'')|$ where we have used the notation $|\cdot - \cdot|$ to denote distance in the appropriate metric spaces. 

%%%% MOVE ME Quasisymmetric equivalence of metric spaces forms the basis of many interesting questions and quasisymmetric uniformization theorems which classify which spaces are quasisymmetrically equivalent to standard spaces, such as $S^2$, are highly desired.

\begin{thm}
\label{QS inv qcap}
Let $Z$ and $W$ be compact, connected, Ahlfors regular metric spaces and let $p > 1$.  If $\varphi \colon Z \to W$ is an $\eta$-quasisymmetric homeomorphism, then there exist $c,C>0$ depending only on $\eta$ and the hyperbolic filling parameters with the following property: whenever $A,B \subseteq Z$ are either open sets with $\dist(A,B) > 0$ or disjoint continua,
\begin{equation*}
c\pcap(\varphi(A),\varphi(B)) \leq \pcap(A,B) \leq C\pcap(\varphi(A), \varphi(B)).
\end{equation*}
\end{thm}

We note here that the $p$ above need not match the Ahlfors regularity dimension of neither $Z$ nor $W$ and that $Z$ and $W$ might have different Ahlfors regularity dimensions.  The quasisymmetric invariance result relies on the fact that a quasisymmetry on compact, connected, metric measure spaces induces a quasi-isometry on corresponding hyperbolic fillings.  A map $F$ between two metric spaces $X$ and $Y$ is said to be a quasi-isometry if there are constants $c, C>0$ such that for all $x, x' \in X$, we have
\begin{equation}
\label{QI def eq}
\frac{1}{C} |x - x'| - c \leq |f(x) - f(x')| \leq C|x - x'| + c
\end{equation}
and there is a constant $D > 0$ such that for all $y \in Y$, there is an $x \in X$ such that $|f(x) - y| \leq D$.

We now define the second quantity: weak covering $p$-capacity ($\wcpcap$).  As before, there is a choice of hyperbolic filling required that introduces a multiplicative ambiguity but which poses no issues for the statements of the theorems.  Unlike $\pcap$, the quantity $\wcpcap$ is defined for all path families in a given metric space.

The vertices $V$ in our hyperbolic filling correspond to balls in $Z$: we let $B_v$ denote the ball corresponding to the vertex $v \in V$.  A subset $S \subseteq V$ is said to cover $Z$ if $Z \subseteq \cup_{v \in S}{B_v}$.   Let $\mathscr{S} = \{S_n\}$ where each $S_n\subseteq V$ is finite and covers $Z$.  We call such an $\mathscr{S}$ a sequence of covers.  We say $\mathscr{S}$ is expanding if for every finite $A \subseteq V$, we have $S_n \cap A = \emptyset$ for all large enough $n$.  

For a given subset $S \subseteq V$ that covers $Z$ and a path $\gamma \colon [0,1] \to Z$ we define a projection $P \colon [0,1] \to V$ onto $S$ as a partition $0=t_0 < t_1 < \dots < t_m = 1$ of $[0,1]$ and a sequence $v_1, \dots v_m \in S$ such that for all $k$, $\gamma([t_{k-1}, t_k]) \subseteq B_{v_k}$.

Given $\tau \colon V \to [0,\infty]$, we define the $\tau$-length of a projection $P$ of $\gamma$ on $S$ as $\ell_{\tau,P,S} (\gamma) = \sum_k \tau(v_k)$.

Now, let $\mathscr{S} = \{S_n\}$ be an expanding sequence of covers.  Given a rectifiable path $\gamma \colon [0,1] \to Z$, we say $\tau$ is admissible for $\gamma$ relative to $\mathscr{S}$ if 
\begin{equation*}
\liminf_{n \to \infty} ( \inf_{P} \ell_{\tau,P,S_n} (\gamma) ) \geq 1
\end{equation*}
where the infimum $\inf_P$ is over all projections of $\gamma$ onto $S_n$.
We say $\tau$ is admissible for $\gamma$ if $\tau$ is admissible relative to $\mathscr{S}$ for all such $\mathscr{S}$.  We remark that, for a given $\gamma$, changing the parameterization does not affect the subsequent projected $\tau$-length and we will frequently view rectifiable $\gamma$ as parameterized by arclength.

The main idea is to use subsets (the covers above) of vertices of the hyperbolic filling to approximate $Z$ in finer and finer resolution.  Path families now lie on the boundary $Z$ and are projected onto these covers in order to test admissibility of functions defined within.  By infimising over all projections onto a given cover and then letting the covers ``expand'' to become more and more like $Z$, we arrive at the rough length a given function defined on the filling gives a particular path.  Demanding admissibility for all rectifiable paths as with $\modp$ and using the weak norm as with $\pcap$ leads us to our definition.

\begin{defn}
Given a collection of paths $\Gamma$ in $Z$, we define the weak covering $p$-capacity $\wcpcap(\Gamma)$ of $\Gamma$ as
\begin{equation*}
\wcpcap(\Gamma) = \inf \{ \norm {\tau}_{p,\infty}^p : \tau \  \text{is admissible for all } \gamma \in \Gamma \}.
\end{equation*}
\end{defn}

With this quantity we have comparability even without the Loewner condition.

\begin{thm}
\label{wcqcap = qmod}
Let $Q>1$ and let $(Z,d,\mu)$ be a compact, connected Ahlfors $Q$-regular metric space.  Then there exist constants $c,C>0$ depending only on $Q$ and the hyperbolic filling parameters such that for all path families $\Gamma$,  
\begin{equation*}
c\qmod(\Gamma) \leq \wcqcap(\Gamma) \leq C\qmod(\Gamma).
\end{equation*}
\end{thm}

Similarly to $\pcap$, the quantity $\wcpcap$ has a quasisymmetric invariance property.
\begin{thm}
\label{QS inv wcqcap}
Let $Z$ and $W$ be compact, connected, Ahlfors regular metric spaces and let $p > 1$.  If $\varphi \colon Z \to W$ is an $\eta$-quasisymmetric homeomorphism, then there exist constants $c,C>0$ depending only on $p$, $\eta$, and the hyperbolic filling parameters such that for all path families $\Gamma$ in $Z$ we have 
\begin{equation*}
c\wcpcap(\Gamma) \leq \wcpcap(\varphi(\Gamma)) \leq C\wcpcap(\Gamma).
\end{equation*}
\end{thm}

This quasisymmetric invariance property implies a result due to Tyson \cite{Ty}.

\begin{cor}
Let $Z$ and $W$ be compact, connected, Ahlfors $Q$-regular metric spaces with $Q>1$ and let $\varphi \colon Z \to W$ be an $\eta$-quasisymmetric homeomorphism.  Then there exists constants $c,C>0$ depending only on $\eta$, $Q$, and the hyperbolic filling parameters such that
\begin{equation*}
c \qmod(\Gamma) \leq \qmod(\varphi(\Gamma)) \leq C \qmod(\Gamma)
\end{equation*}
for all path families $\Gamma$ in $Z$.
\end{cor}

Indeed, this follows immediately from combining Theorem \ref{wcqcap = qmod} and Theorem \ref{QS inv wcqcap} above.  Tyson \cite{Ty} shows this result for locally compact, connected, Ahlfors $Q$-regular metric spaces, but our framework with hyperbolic fillings is adapted to compact metric spaces.  Williams \cite{Wi} also derives this result; his Remark $4.3$ relates the conditions in the corollary above to his condition (III) which leads to the conclusion.

Lastly we outline the structure of the paper.  In Section \ref{Sec Preliminaries} we review some preliminaries including the definition of modulus and the statement of a weak $\ell^p$-norm comparison result. In Section \ref{Sec Hyp Fill} we construct hyperbolic fillings and prove some of their basic properties.  In particular, we sketch the proof that a quasisymmetry between boundary spaces induces a quasi-isometry between the corresponding hyperbolic fillings.  In Section \ref{Sec qcap} we prove the results related to $\pcap$, namely Theorems \ref{qmod < qcap}, \ref{qcap < qmod}, and \ref{QS inv qcap}.  In Section \ref{Cdmin} we relate $\pcap$ to Ahlfors regular conformal dimension.  In Section \ref{Sec wcqcap} we prove the results corresponding to $\wcpcap$, namely Theorems \ref{wcqcap = qmod} and \ref{QS inv wcqcap}.

\begin{ack}
The author thanks Mario Bonk for countless interesting discussions related to the contents of this paper; his input and guidance have been immensely valuable.  The author also thanks Peter Ha{\"i}ssinsky and Kyle Kinneberg for stimulating conversations on the subject matter.
\end{ack}

% SECTION 2

\section{Preliminaries}
\label{Sec Preliminaries}

First we define the notion of modulus of path families.  Let $(Z, d, \mu)$ be a metric measure space.  By a path in $Z$ we mean a continuous function $\gamma: I \to Z$ where $I \subseteq \R$ is an interval.  We will use $\gamma$ to refer to both the path and the image of the path.  For a path family $\Gamma$, we say a Borel function $\rho \colon Z \to [0,\infty]$ is admissible for $\Gamma$ if for all rectifiable $\gamma \in \Gamma$ we have
\begin{equation*}
\int_\gamma \rho \geq 1.
\end{equation*} 
Here and elsewhere path integrals are assumed to be with respect to the arclength parameterization.  We define the $p$-modulus of the path family $\Gamma$ to be
\begin{equation*}
\modp(\Gamma) = \inf \left\{ \int_Z \rho^p \right \}
\end{equation*}
where the infimum is taken over all $\rho$ admissible for $\Gamma$ and the integral is against the measure $\mu$.  In the following we will also suppress $d \mu$ from the notation.

One may think of $\modp$ as an outer measure on path families which is supported on rectifiable paths.  The quantity $\modp$ is meaningless in spaces without rectifiable paths such as the snowflaked metric spaces discussed in the introduction.  Nonetheless, where rectifiable paths abound $\modp$ is closely related to conformality and quasiconformality.  Indeed, a conformal diffeomorphism $f$ between two Riemannian manifolds $M$ and $N$ of dimension $n$ preserves the $n$-modulus of path families \cite[Theorem 7.10]{He}.  This invariance principle can be used to give an equivalent definition of quasiconformal homeomorphisms between appropriate spaces \cite[Definition 7.12]{He} which, while difficult to check, is quite strong.  Many basic properties of modulus are detailed in \cite[Chapter 7]{He}.

Often for clarity and brevity we will make use of the symbols $\lesssim$ and $\simeq$.  For two quantities $A$ and $B$, that may depend on some ambient parameters, we write $A \lesssim B$ to indicate that there is a constant $C>0$ depending only on these parameters such that $A \leq CB$. We also write $A \simeq B$ to indicate that there are constants $c, C>0$ depending only on these parameters such that $cB \leq A \leq CB$.  The exact dependencies of these constants will be clear from the context.  

For some results we will make use of \cite[Lemma 2.2]{BoS}.  For convenience we include the statement here.  

\begin{lemma}
\label{BoS Lemma}
Let $H$ and $K$ be countable sets.  Let $p > 1$ and $J \subseteq H \times K$ be such that all sets of the forms $J_h = \{k \in K : (h,k) \in J\}$ and $J^k = \{h \in H : (h,k) \in J\}$ have at most $N$ elements.  Then there is a constant $C(p,N)>0$ with the following weak $\ell^p$ bound property: if $s = \{s_h\}$ and $t = \{t_k\}$ are sets of real numbers indexed by $H$ and $K$ respectively and
\begin{equation}
\label{BoS Lemma sum condition}
|s_h| \leq \sum_{k \in J_h} |t_k|
\end{equation}
for all $h$, then $\norm{s}_{p,\infty} \leq C(p,N) \norm{t}_{p, \infty}$.  
\end{lemma}

One useful notion for a metric space to have many rectifiable paths is the Loewner condition.  This relates the modulus of path families connecting nonintersecting continua, say $A$ and $B$, with their relative distance 
\begin{equation*}
\Delta(A,B) = \frac{\dist(A,B)}{\min\{\diam(A), \diam(B)\}}.
\end{equation*}
We let $\qmod(A,B)$ denote the modulus of the path family connecting $A$ to $B$. We say a metric measure space $(Z, d, \mu)$ is a $Q$-Loewner space if there is a decreasing function $\phi_Q \colon (0,\infty) \to (0,\infty)$ such that for all such $A$ and $B$ we have
\begin{equation*}
\phi_Q (\Delta(A,B)) \leq \qmod(A,B).
\end{equation*}
See  \cite[Chapter 8]{He} for this definition and more information on Loewner spaces.

The main intuition here is that the path family connecting two continua with positive relative distance is large enough to carry positive $Q$-modulus; that is, there are many rectifiable paths connecting the two sets.  This condition cannot be omitted from Theorem \ref{qcap < qmod}: by the quasisymmetric invariance principle for $\qcap$ (Theorem \ref{QS inv qcap}) we can snowflake a $Q$-Loewner space to construct spaces where disjoint open sets have positive $\qcap$ but the modulus of any path family is 0.

Given a function $u$ on a metric measure space $(Z,d,\mu)$, we say that a Borel function $\rho \colon Z \to [0,\infty]$ is an upper gradient for $u$ if whenever $z,z' \in Z$ and $\gamma$ is a rectifiable path connecting $z$ and $z'$, we have
\begin{equation*}
|u(z) - u(z')| \leq \int_\gamma \rho.
\end{equation*}
We use the notation that if $B = B(x,r)$, then $\lambda B = B(x, \lambda r)$.  For a given ball $B$ and a locally integrable function $u$ we set $u_B = \frac{1}{\mu(B)} \int_B u = \dashint_B{u}$; this is the average value of $u$ over the ball $B$.  We say $(Z, d, \mu)$ admits a $p$-Poincar\'e inequality if there are constants $C>0$ and $\lambda \geq 1$ such that
\begin{equation}
\label{Poincare Inequality}
\dashint_{B} |u - u_{B}| \leq C (\diam B)\left( \dashint_{\lambda B} \rho^p \right)^{1/p}
\end{equation}
for all open balls $B$ in $Z$, for every locally integrable function $u \colon Z \to \R$, and every upper gradient $\rho$ of $u$ in $Z$.   This definition and a subsequent discussion can be found in \cite[Chapter 8]{HKST}.  The $Q$-Poincar\'e inequality is a regularity condition on our space which will follow from the $Q$-Loewner space hypothesis present in some of the theorems.

%We will use \cite{HK} to show that under our conditions in a Loewner space we are also in a space admitting a weak $(1,Q)$-Poincar\'e inequality. Using \cite{HKST} we will be able to use measurable functions $u$ instead of just continuous functions. We also lower the exponent $Q$ in the Poincar\'e inequality to some $p < Q$.  This self improvement of the Poincar\'e inequality is due to \cite{KZ}.

We make use of Hausdorff content and Hausdorff measure.  Given an exponent $\alpha \geq 0$ and a set $E$, the $\alpha$-Hausdorff content of $E$ is defined as
\begin{equation*}
\mathscr{H}_\infty^\alpha(E) = \inf \{ \sum_{A \in \mathscr{A}} \diam(A)^\alpha : \mathscr{A} \text{ is a cover of } E  \text{ by balls } A\}.
\end{equation*}
The $\alpha$-Hausdorff measure of $E$ is defined as $\mathscr{H}^\alpha(E) = \lim_{\delta \to 0} \mathscr{H}_\delta^\alpha(E)$ where
\begin{equation*}
\mathscr{H}_\delta^\alpha(E) = \inf \{ \sum_{A \in \mathscr{A}} \diam(A)^\alpha : \mathscr{A} \text{ is a cover of } E \text{ by balls A with } \diam(A) < \delta \}.
\end{equation*}
The Hausdorff dimension of $E$ is defined as $\dimh(E) = \inf\{\alpha : \mathscr{H}^\alpha(E) = 0\}$.

% First, we note that admissibility relies on rectifiable paths.  In particular, if $\Gamma$ has no rectifiable curves, then $\qmod(\Gamma) = 0$.  As such, $\qmod(\emptyset) = 0$.  Second, if $\Gamma_1, \Gamma_2$ are curve families such that every $\gamma \in \Gamma_1$ also satisfies $\gamma \in \Gamma_2$, then $\qmod(\Gamma_1) \leq \qmod(\Gamma_2)$ as every function which is admissible for $\Gamma_2$ is also admissible for $\Gamma_1$.  Lastly, if $\Gamma, \Gamma'$ are curve families such that every curve in $\gamma$ has a subcurve $\gamma'$ in $\Gamma$, then $\qmod(\Gamma) \leq \qmod(\Gamma')$ as every function which is admissible for $\Gamma'$ is admissible for $\Gamma$.

% SECTION 3

\section{Hyperbolic fillings}
\label{Sec Hyp Fill}

Here we detail the construction of the hyperbolic fillings used to define $\qcap$ and $\wcqcap$.  The main idea is to construct a graph that encodes the combinatorial data of $Z$ in finer and finer detail.  As remarked in the introduction, it is useful to keep a Whitney cube decomposition of the unit disk model of the hyperbolic plane $\mathbb{H}^2$ in mind with the vertices corresponding to the Whitney cubes and with edges existing between intersecting cubes.

Recall we work in a compact, connected, Ahlfors $Q$-regular metric measure space $(Z,d,\mu)$.  Our construction of and proofs of properties involving the hyperbolic fillings of $Z$ follows \cite[Section 2.1]{BP} almost exactly.  There is a minor problem with their construction, however, which we fix by using doubled radii balls inspired a similar construction in \cite[Section 6.1]{BuS}.  For completeness we include the entire construction here.  

By rescaling, we may assume $\diam Z < 1$.  Let $s>1$ and, for each $k \in \N_0$, let $P_k \subseteq Z$ be an $s^{-k}$ separated set that is maximal relative to inclusion.  We call $s$ the parameter of the hyperbolic filling.  To each $p \in P_k$, we associate the ball $v = B(p, 2s^{-k})$ which we will use as our vertices in our graph.  We refer to $k$ as the level of $v$ and write this as $\ell(v)$.  We also write $V_k$ for the set of vertices with level $k$.  Note as $\diam(Z) < 1$, there is a unique vertex in $V_0$ which we will denote $O$.  We often write $v$ as $B_v$ or $B(v)$ where the level $k$ and the center $p$ are understood.  We will occasionally make an abuse of notation, however, and write $B(v,2s^{-k})$ referring to $v$ as both the vertex and the center of the ball.  We also use the notation $r(B)$ for the radius of a ball. 

We form a graph $X=(V,E)$ where the vertex set $V$ is the disjoint union of the $V_k$ and we connect two distinct $v,w \in V$ by an edge if and only if $|\ell(v) - \ell(w)| \leq 1$, and $\overline{B_v} \cap \overline{B_w} \neq \emptyset$.  We endow $X$ with the unique path metric in which each edge is isometric to an interval of length 1.    
    
For $v,w \in V$, we let $(v,w)$ denote the Gromov product given by
\begin{equation*}
(v,w) = \frac{1}{2}(|O - v| + |O - w| - |v - w|),
\end{equation*}
where $O$ is the unique vertex with $\ell(O) = 0$ and $|v - w|$ is the graph distance between $v$ and $w$.  Intuitively, $(v,w)$ is roughly the distance between $O$ and any geodesic connecting $v$ to $w$.
\begin{lemma}
\label{Gromov Comp}
For $v,w \in V$, we have $s^{-(v,w)} \simeq \diam(B_v \cup B_w)$.
\end{lemma}

\begin{proof}
We find $x \in V$ with $\diam(B_x) \simeq \diam(B_v \cup B_w)$ and $B_v, B_w \subseteq B_x$.  To do this, we note that there is a number $k \in \N$ with $s^{-k-1} \leq \diam(B_v \cup B_w) \leq s^{-k}$.  Then by construction there is a vertex $x$ on level $k$ such that $\frac{1}{2}B_x \cap B_v \neq \emptyset$; this follows from the choice of an $s^{-k}$ separated set for the centers of the balls corresponding to the vertices on level $k$.  As $\diam(B_v \cup B_w) \leq s^{-k}$, we see $B_v \cup B_w \subseteq B_x$.  Without the extra radius factor this inclusion need not be true; this is the minor oversight in \cite[Section 2.1]{BP}.  We choose geodesics $[Ov]$ and $[Ow]$ containing $x$.  We see $(v,w) \geq |O-x|$ as 
\begin{equation*}
\begin{split}
(v,w) - |O-x| &= \frac{1}{2}( |O-v| - |O-x| + |O-w| - |O-x| - |v-w|) \\
&= \frac{1}{2} (|v-x| + |w-x| - |v-w|)
\end{split}
\end{equation*}
which is nonnegative by the triangle inequality.  Thus,
\begin{equation*}
s^{-(v,w)} \leq s^{-|O-x|} = \frac{1}{2}r(B_x) \simeq \diam(B_v \cup B_w).
\end{equation*}
For the other direction, we follow the notation in \cite[Lemma 2.2]{BP} and set $|v-w|=\ell, |O-v|=m,$ and $|O-w|=n$. Let $[vw]$ be a geodesic segment, which is formed from a sequence of balls $B_k$ for $k \in {0,\dots,\ell}$ with $B_0 = B_v$ and 
\begin{equation*}
s^{-1}r(B_i) \leq r(B_{i+1}) \leq s r(B_i)
\end{equation*}
 for all $i$.  Hence, for $k \in {0,\dots,\ell}$, we have
 \begin{equation*}
 \begin{split}
 \diam(B_v \cup B_w) &\leq \sum_{i=0}^\ell \diam(B_i) \\
 &= \sum_{i=0}^k \diam(B_i) + \sum_{j=0}^{\ell - k - 1} \diam(B_{\ell - j})\\
 &\leq 4 \sum_{i=0}^k s^{-m+i} + 4 \sum_{j=0}^{\ell - k - 1} s^{-n+j}\\
 &= \frac{4}{s-1} (s^{-m+k+1} - s^{-m} + s^{-n+\ell-k} - s^{-n} ) \\
 &\leq \frac{4s}{s-1}(s^{-m+k} + s^{-n+\ell-k}).
 \end{split}
 \end{equation*}
Setting $k = \frac{1}{2}(\ell+m-n)$ (or $k = \frac{1}{2}(\ell+m-n + 1)$ for a comparable bound if this is not an integer), this becomes
\begin{equation*}
\diam(B_v \cup B_w) \leq \frac{8s}{s-1} s^{\frac{1}{2} (\ell - m - n)} = \frac{8s}{s-1} s^{-(v,w)}.
\qedhere
\end{equation*} 
\end{proof}

The following lemma involves Gromov hyperbolic metric spaces.  For definitions we refer the reader to \cite[Section 2.1]{BuS}.

\begin{lemma}
$X$ equipped with the graph metric is a Gromov hyperbolic space.
\end{lemma}

\begin{proof}
Let $v,w,x \in V$.  Then
\begin{equation*}
\diam(B_v \cup B_w) \leq \diam(B_v \cup B_x) + \diam(B_x \cup B_w),
\end{equation*}
so by the above lemma there is a constant $D>0$ independent of $v,w,$ and $x$ such that 
\begin{equation*}
s^{-(v,w)} \leq D (s^{-(v,x)} + s^{-(x,w)}) \leq 2D \max(s^{-(v,x)}, s^{-(x,w)}).
\end{equation*}
Hence, 
\begin{equation*}
-(v,w) \leq \log_s(2D) + \max(-(v,x),-(x,w))
\end{equation*}
and so
\begin{equation*}
(v,w) \geq \min((v,x),(x,w)) - \log_s(2D)
\end{equation*}
which is the inequality required in the definition of a Gromov hyperbolic space.
\end{proof}

We will also work with the boundary at infinity of our hyperbolic fillings.  For completeness we include the standard construction of the boundary at infinity of a Gromov hyperbolic space here and refer the reader to \cite[Chapter 2]{BuS} for some of the details as well as more background.  

For our given Gromov hyperbolic space $X$, the points of the boundary at infinity $\partial_\infty X$ are equivalence classes of sequences of points ``diverging to infinity''.  More precisely, we say a sequence of points $\{ x_n \}$ diverges to infinity if
\begin{equation*}
\lim_{m,n \to \infty} (x_n, x_m) = \infty.
\end{equation*}
Two sequences $\{ x_n \}$ and $\{ y_m \}$ are equivalent if
\begin{equation*}
\lim_{m,n \to \infty} (x_n, y_m) = \infty.
\end{equation*}
One then may extend the Gromov product to the boundary as in \cite{BuS}.  From this one defines a metric $d$ on $\partial_\infty X$ to be a visual metric if there are constants $c, C>0$ and $a>1$ such that for all $z, z' \in \partial_\infty X$ we have
\begin{equation*}
ca^{-(z, z')} \leq d(z,z') \leq Ca^{-(z,z')}.
\end{equation*}

We now relate this construction to $X$ and $Z$.

\begin{lemma}  With our constructed $X$ above, we can identify $\partial_\infty X$ with $Z$ where the original metric on $Z$ is a visual metric.
\end{lemma}

 \begin{proof}[Sketch of proof]
We use 
\begin{equation}
\label{s sim diam eq}
s^{-(v,w)} \simeq \diam(B_v \cup B_w).
\end{equation}
We see $\{v_n\}$ is a sequence of vertices diverging to infinity if and only if
\begin{equation*}
\diam(B_{v_n} \cup B_{v_m}) \to 0
\end{equation*}
and so not only do the diameters satisfy $\diam(B_{v_n}) \to 0$ but the centers $p_n$ of the balls corresponding to $v_n$ also converge to a single point in $z \in Z$.  This shows we can view $\partial_\infty X$ as a subset of $Z$ by identifying a sequence of vertices diverging to infinity with the limit point of the centers of the corresponding balls.  The other inclusion follows by considering that for each $k \in \N_0$ the set of balls $\{B_v : v \in V_k\}$ covers $Z$.  Hence, for fixed $z \in Z$ for each $k \in \N_0$ we may choose a vertex $v_k$ with level $k$ such that $z \in B_{v_k}$.  This creates a sequence in $X$ diverging to infinity that corresponds to $z$.  Relation \eqref{s sim diam eq} above also shows our original metric on $Z$ is in fact a visual metric with respect to $X$.
\end{proof}

The metric paths in $X$ that we are interested in travel through many vertices and are often infinite.  For this reason we define a path in $X$ as a (possibly finite) sequence of vertices ${v_k}$ such that for all $k$, the vertices $v_k$ and $v_{k+1}$ are connected by an edge.  Alternatively we may view a path as a sequence of edges ${e_k}$ such that for all $k$, the edges $e_k$ and $e_{k+1}$ share a common vertex; the point of view will be clear from context.

We now specify what it means for a sequence of vertices $v_n$ to converge to $z \in Z$: if $v_n$ is represented by $B(p_n, r_n)$ then $v_n \to z$ if and only if $p_n \to z$ and $r_n \to 0$.  From this we also see what it means for a path (given by a sequence of vertices $\{v_k\}_{k \in \N}$ or edges $\{e_k\}_{k \in \N}$ as discussed above) to converge to a point in $Z$.

In our definitions we will work with nontangential limits.  Intuitively, a path approaches the boundary nontangentially if it stays within bounded distance of a geodesic.  In our setting this means that the smaller the radii corresponding to vertices on a path are, the closer the centers corresponding to those vertices need to be to the limit point.  We state this more precisely as a definition.

\begin{defn}
\label{nontangential limits}
A path in $X$ with vertices $v_n$ represented by $B(p_n, r_n)$ converges nontangentially to $z \in Z$ if $p_n \to z$ and $r_n \to 0$ and there exists a constant $C>0$ such that for all $n \in \N$ we have $\dist(z,p_n) \leq C r_n$.
\end{defn}

For the next lemma, we define the valence of a vertex $v$ in a graph $(V,E)$ as the number of edges having $v$ as a vertex.  To say a graph has bounded valence then means that there is a uniform constant $C$ such that every vertex $v\in V$ has valence at most $C$.

\begin{lemma}
\label{bounded valence}
The hyperbolic filling $X$ of a compact, connected, Ahlfors $Q$-regular metric measure space has bounded valence.
\end{lemma}

\begin{proof}
Let $v$ be a vertex with level $n > 1$.  Let $W$ be the vertices with level $n$ that intersect $v$.  We bound $|W|$, the cardinality of the set $W$.   Bounds on the number of vertices adjacent to $v$ with levels $n-1$ and $n+1$ follow similarly and yield the result.  

Here we abuse notation and use $v$, $w$ as the centers of the balls corresponding to these vertices.  We note $\cup_{w \in W}B(w,\frac{1}{3}s^{-n})$ is a disjoint collection of balls by the separation of vertices on level $n$.  Moreover, this collection is contained in $B(v,4s^{-n})$.  By Ahlfors regularity there are constants $c,C$ such that for all $z \in Z$ and $r \in (0, \diam(Z)]$ we have 
\begin{equation*}
cr^{Q} \leq \mu(B(z,r)) \leq Cr^{Q}.
\end{equation*}
Hence, we see
\begin{equation*}
c |W| \frac{1}{3^Q} s^{-nQ} \leq C 4^Q s^{-nQ}
\end{equation*}
and so $|W|$ is uniformly bounded above.
\end{proof}

Our final lemma is a construction which extends a quasisymmetric homeomorphism $f$ between two compact, connected, metric measure spaces $Z$ and $W$ to a quasi-isometric map $F$ between corresponding hyperbolic fillings $X$ and $Y$.  As we make use of this construction explicitly in Theorem \ref{QS inv wcqcap}, we record the result here as a lemma as well as a sketch of the proof.  The full proof can be found in \cite[Theorem 7.2.1]{BuS}.  

\begin{lemma}
\label{QI induces QS}
Let $Z$ and $W$ be compact, connected, metric measure spaces with corresponding hyperbolic fillings $X$ and $Y$.  Let $f \colon Z \to W$ be a quasisymmetric homeomorphism.  Then there is a quasi-isometry $F \colon X \to Y$ which extends to $f$ on the boundary.
\end{lemma}

\begin{proof}[Sketch of proof]
It suffices to define $F$ as a map between vertices.  Given $x \in X$, we see $f(B_x) \subset W$ and so there is at least one vertex $y \in Y$ of minimal radius (i.e. highest level) that contains $f(B_x)$.  We set $F(x) = y$.  We show $F$ is a quasi-isometry.  For the upper bound in (\ref{QI def eq}) we show that for vertices $x, x' \in X$ with $|x-x'| \leq 1$, we have uniform control over $|F(x) - F(x')|$.  If not, then we have a sequence of such $x,x'$ such that $|F(x) - F(x')|$ gets arbitarily large.  With this, we note $\overline{B}_{F(x)} \cap \overline{B}_{F(x')} \neq \emptyset$ and from this deduce that the levels between, and hence the ratio of the radii  of the balls $B_{F(x)}$ and $B_{F(x')}$, must become arbitrarily large.   Using a common point in $\overline{B}_{F(x)} \cap \overline{B}_{F(x')}$ and the quasisymmetry condition, we arrive at a contradiction.  For the lower bound  one considers $G \colon Y \to X$ defined as above but using $f^{-1}$ in place of $f$.  One then shows that the compositions $G \circ F$ and $F \circ G$ are within a bounded distance from the identity in each case (i.e. there is a $D > 0$ such that for all $x \in X$ one has $|x - (G \circ F)(x)| < D$ and likewise for $F \circ G$).  The fact that $F(X)$ is at a bounded distance from any point in $y$  also follows from the compositions being within a bounded distance from the identity.
\end{proof}

%SECTION WEAK Q CAP

\section{Weak capacity}
\label{Sec qcap}

%%%%%
\iffalse
Here we prove some basic properties and then the main theorems about $\qcap$.

\begin{lemma}
The quantity $\qcap$ has the following properties: \smallskip  \\ \smallskip
$\mathrm{(i)}$ If $A \subseteq A'$ and $B \subseteq B'$, then $\qcap(A,B) \leq \qcap(A',B')$ \\
$\mathrm{(ii)}$ If $A_k$ and $B$ are disjoint with $\dist(\cup A_k,B) > 0$, then 

\begin{equation*}
\qcap(\cup A_k, B) \leq \sum \qcap(A_k, B).
\end{equation*}
\end{lemma}

\begin{proof}
Property (i) follows as any function admissible for $\qcap(A',B')$ is also admissible for $\qcap(A,B)$.
Property (ii) follows as if $\tau_k$ is admissible for $\qcap(A_k,B)$, then $\sup_k \tau_k$ is admissible for $\qcap(\cup A_k, B)$.
\end{proof}
\fi

Now we prove the main results involving weak capacity: Theorems \ref{qmod < qcap}, \ref{qcap < qmod}, and \ref{QS inv qcap}.  Recall we consider compact, connected, Ahlfors $Q$-regular metric measure spaces $(Z,d,\mu)$.  After the proofs of the main theorems we prove that for open $A,B \subseteq Z$ with $\dist(A,B) > 0$ we always have $\pcap(A,B) > 0$.  We prove Theorem \ref{qmod < qcap} for disjoint continua first.  For this we need a technical lemma.

\begin{lemma}
\label{short path lemma}
Let $p \geq 1$, let $f \in \ell^p(V)$, and let $A \subseteq Z$ be a continuum.  Let $\{B_{v_n}\}$ be a sequence of balls corresponding to vertices $v_n$ with $\ell(v_n) \to \infty$.  Suppose that there is a constant $c>0$ such that for all large enough $n$ we have
\begin{equation*}
\Hc(A \cap B_{v_n}) \geq c s^{-\ell(v_n)}.
\end{equation*}
Then for each $\delta > 0$ there is an $N$ such that for all $n \geq N$, there is a vertex path $\sigma_n \subseteq X$ that starts from $v_n$ and has limit in $A$ which satisfies $\sum_{\sigma_n} f(v) \leq \delta$.
\end{lemma}

\begin{proof}
We note if $f$ has finite support then the result is immediate.  Hence, we assume without loss of generality that $f$ does not have finite support.  We also may assume all balls in consideration have radius bounded above by 4 as $\diam(Z) \leq 1$.  For $\epsilon > 0$ and fixed $n$ we define
\begin{equation*}
K_n = K_n(\epsilon) = \{ z \in A \colon \exists v \in X \text{ with } z \in B_v, f(v)^p \geq \epsilon r(B_v) \text{ and } \ell(v) \geq \ell(v_n)\}.
\end{equation*}
We then may cover $K_n$ by $\{B_{v_z}\}_{z \in K_n}$ where $f(v_z)^p \geq \epsilon r(v_z)$.  As the balls corresponding to these vertices have uniformly bounded diameter there is a countable subcollection of vertices $\mathscr{G}_n$ such that 
\begin{equation*}
\bigcup_{z \in K_n}B_{v_z} \subseteq \bigcup_{v \in \mathscr{G}_n} 5B_v
\end{equation*}
and such that if $v,w \in \mathscr{G}_n$ are distinct, then $B_v \cap B_w = \emptyset$ (see \cite[Theorem 1.2]{He}).  Hence, $K_n \subseteq \cup_{v \in \mathscr{G}_n} 5B_v$. Thus, we have
\begin{equation*}
\Hc(K_n) \leq \sum_{\mathscr{G}_n} 10 r(v) \leq 10 \sum_{\mathscr{G}_n} \frac{f(v)^p}{\epsilon} \leq \frac{10}{\epsilon} \norm{f_{\ell(v_n)}}_p^p
\end{equation*}
where $f_j$ is $f$ restricted to vertices of level $j$ and higher.

In the above we use $\epsilon = \epsilon_n$ defined by
\begin{equation*}
\frac{10}{\epsilon_n} \norm{f_{\ell(v_n)}}_p^p  = \frac{1}{2} cs^{-\ell(v_n)}
\end{equation*}
which is possible as $\norm{f_{\ell(v_n)}}_p^p > 0$ for all $n$ as $f$ does not have finite support.  We conclude that for large enough $n$ we have 
\begin{equation*}
\Hc(K_n) \leq \frac{1}{2} \Hc(A \cap B_{v_n}).
\end{equation*}
Hence, there is a point $z \in (A \cap B_{v_n}) \setminus K_n$.

For this $z$ we form a vertex path $\sigma$ from $v_n$ with limit $z$ by choosing a vertex $w_k$ for each $k > \ell(v_n)$ such that $z \in B_{w_k}$.  As $z \notin K_n$, we estimate the sum
\begin{equation*}
\begin{split}
\sum_\sigma f(v) &\leq (\epsilon_n^{1/p}) \sum_{k=\ell(v_n)}^\infty (2s^{-k})^{1/p} \\
&= (2 \epsilon_n)^{1/p} \sum_{k=\ell(v_n)}^\infty (s^{1/p})^{-k} \\
&\lesssim (2 \epsilon_n)^{1/p} s^{-\ell(v_n)/p} \\
&\lesssim (\epsilon_n s^{-\ell(v_n)})^{1/p}.
\end{split}
\end{equation*}
From our choice of $\epsilon_n$ we have
\begin{equation*}
\norm{f_{\ell(v_n)}}_p^p  \simeq \epsilon_n s^{-\ell(v_n)}
\end{equation*}
and as $\norm{f_{\ell(v_n)}}_p^p \to 0$ as $n \to \infty$, the result holds.
\end{proof}

%For an open set $A \subseteq Z$ and $\lambda > 0$, we define $ A_\lambda = \{a \in A : d(a, Z \setminus A) > \lambda \}$.

\begin{proof}[Proof of Theorem \ref{qmod < qcap}]
%We first show that for all $\lambda > 0$ for which $A_\lambda$ and $B_\lambda$ are open we have $\qmod(A_\lambda, B_\lambda) \lesssim \qcap(A,B)$ with a constant that does not depend on $\lambda$. 
We begin by showing the result for disjoint continua $A$ and $B$.  Let $\tau \colon E \to [0,\infty]$ be an admissible function for $\qcap(A,B)$.  Define $f \colon V \to \R$ by $f(v)=\sum_{e \sim v} \tau(e)$ where the sum is over all edges with $v$ as an endpoint.

We note $\norm{f}_{Q,\infty}^Q \simeq \norm{\tau}_{Q,\infty}^Q$ follows from Lemma \ref{BoS Lemma}.  Indeed, in this case our set $J \subseteq V \times E$ consists of vertex-edge pairs $(v,e)$ with $e$ having $v$ as a boundary vertex where $s_v = f(v)$ and $t_e = \tau(e)$.  As $X$ has bounded valence, Lemma \ref{BoS Lemma} applies to give us one direction of comparability.  By using edge-vertex pairs $(e,v)$ and adjusting the definitions of the $s_v$ and $t_e$ we also get the other bound.  We also note that if $\norm{f}_{Q,\infty} < \infty$ this means $f \in \ell^p$ whenever $p > Q$.

Consider the functions $u_n \colon Z \to \R$ given by 
\begin{equation}
\label{un def}
u_n = \sum_{v \in V_n} \frac{f(v)}{r(B_v)}\chi_{2B_v},
\end{equation}
where we recall $V_n$ are the vertices on level $n$.  We claim $2 u_n$ is admissible for $Q$-modulus between $A$ and $B$ for large enough $n$.  Suppose this fails for some sequence $n_i$.  Then there is a rectifiable path $\gamma_{n_i}$ connecting $A$ and $B$ with $\int_{\gamma_{n_i}} u_{n_i} \leq \frac{1}{2}$.  The endpoints of $\gamma_{n_i}$ lie in balls $\frac{1}{2} B_{v_{n_i}^A}$ and $ \frac{1}{2}B_{v_{n_i}^B}$ corresponding to vertices of level ${n_i}$.  We now apply Lemma \ref{short path lemma} with $\delta = \frac{1}{16}$ to obtain, for all large enough ${n_i}$, paths $\sigma_{n_i}^A$ and $\sigma_{n_i}^B$ with
\begin{equation*}
\sum_{\sigma_{n_i}^A} f(v) < \frac{1}{16} \text { and }\sum_{\sigma_{n_i}^B} f(v) < \frac{1}{16}.
\end{equation*}
From the definition of $f$ it is clear that summing along the edges of these paths gives a $\tau$-length of less than $\frac{1}{16}$ as well.  We note the hypothesis in Lemma \ref{short path lemma} follows once $s^{-{n_i}} < \min(\diam(A), \diam(B))$.

We show this violates the admisibility of $\tau$.  Indeed, if $v_0,\dots,v_m$ is a path of vertices in $V_{n_i}$ where $\gamma_{n_i}$ passes through each $\frac{1}{2} B_{v_k}$ and $\ell(\gamma_{n_i}) \geq 8s^{-{n_i}}$, then
\begin{equation*}
\begin{split}
\int_{\gamma_{n_i}}u_{n_i} &\geq \sum_{j=0}^{m} \frac{r(B_{v_j}) f(v_j)}{2r(B_{v_j})} \\
&= \sum_{j=0}^{m} \frac{1}{2}f(v_j) \\
&\geq \sum_{j=1}^{m} \tau(e_j),
\end{split}
\end{equation*}
where $e_j$ denotes the edge connecting $v_{j-1}$ to $v_j$. Choosing such a path with $v_0$ and $v_m$ corresponding to $B_{v_{n_i}^A}$ and $B_{v_{n_i}^B}$ and combining this path with $\sigma_{n_i}^A$ and $\sigma_{n_i}^B$ yields a path in the graph with nontangential boundary limits in $A$ and $B$ with $\sum \tau(e) < 1$, contradicting the admissibility of $\tau$.  

We have just shown that for large enough $n$ the function $2u_n$ is admissible for $\qmod(A, B)$.  It remains to compute $\norm{u_n}_Q^Q$.  We have 
\begin{equation}
\label{u_n bound}
\begin{split}
\norm{u_n}_Q^Q &= \int_Z \left (\sum_{v \in V_n} \frac{f(v)}{r(B_v)} \chi_{2B_v} \right)^Q \\
&\lesssim  \int_Z \left (\sum_{v \in V_n} \frac{f(v)^Q}{r(B_v)^Q} \chi_{2B_v} \right) \\
&\lesssim \sum_{v \in V_n} f(v)^Q
\end{split}
\end{equation}
where we have used the bounded valence of our hyperbolic filling (Lemma \ref{bounded valence}) for the first inequality and Ahlfors $Q$-regularity for the second.  From \cite[Proof of Theorem 1.4]{BoS}, for any $N>1$ there is an $n \in [N,2N]$ such that 
\begin{equation*}
\sum_{v \in V_n} f(v)^Q \lesssim \norm{f}_{Q,\infty}^Q.
\end{equation*}
 Thus, for specific large enough $n$, we see that $u_n$ is admissible for the modulus between $A$ and $B$ and satisfies $\norm{u_n}_Q^Q \lesssim \norm{\tau}_{Q,\infty}^Q$.  Infimizing over admissible $\tau$ yields the result for continua.

Now, for open sets $A$ and $B$ we use the same technique, but the setup is more involved: we need to work safely inside the open sets to be able to satisfy the hypothesis of Lemma \ref{short path lemma}.  For $\lambda > 0$ define $ A_\lambda = \{a \in A : d(a, Z \setminus A) > \lambda \}$ and define $B_\lambda$ similarly.  Fix $\lambda>0$ such that $A_\lambda$ and $B_\lambda$ are nonempty (all small enough $\lambda$ will satisfy this).  We claim for such $\lambda$ we have $\qmod(A_\lambda,B_\lambda) \lesssim \qcap(A, B)$ with an implicit constant independent of $\lambda$.  Indeed, as above we claim $2u_n$ is admissible for $\qmod(A_\lambda, B_\lambda)$ for large enough $n$, where we recall $u_n$ is defined in equation (\ref{un def}).

If $2u_n$ is not admissible, then we may find a path $\gamma_n$ on level $n$ with $\int_{\gamma_n} u_n \leq \frac{1}{2}$.  We note for large enough $n$, the endpoints $v_n^A$ and $v_n^B$ of $\gamma_n$ satisfy $B_{v_n^A} \subseteq A$ and  $B_{v_n^B} \subseteq B$. We then apply the above procedure to create a short $\tau$-path connecting $A$ and $B$ which contradicts the admissibility of $\tau$.  From the norm computation above, by infimizing over admissible $\tau$ we conclude $\qmod(A_\lambda,B_\lambda) \lesssim \qcap(A, B)$ with an implicit constant independent of $\lambda$.

We show now that this implies $\qmod(A,B) \lesssim \qcap(A,B)$.  %For ease of notation we assume $\norm{u_n}_Q^Q \lesssim \norm{\tau}_{Q,\infty}^Q$ holds for each $n$; to make this precise one must pass to a subsequence. 
 Let $\lambda_n = \frac{1}{n}$ and, for each $n$, let $\sigma_n \geq 0$ be an admissible function for $\qmod(A_{\lambda_n},B_{\lambda_n})$ such that $\norm{\sigma_n}_Q^Q \lesssim \qcap(A,B)$.  

By Mazur's Lemma \cite[p.~19]{HKST}, there exist convex combinations $\rho_n$ of $\sigma_k$ with $k \geq n$ and a limit function $\rho$ such that $\rho_n \to \rho$ in $L^Q$.  By Fuglede's Lemma \cite[p.~131]{HKST}, after passing to a subsequence of the $\rho_n$, we may assume that for all paths $\gamma$ except in a family $\Gamma_0$ of $Q$-modulus 0, $\int_{\gamma} \rho_n \to \int_{\gamma} \rho$.  As the $Q$-modulus of $\Gamma_0$ is $0$, there exists a function $\sigma \geq 0$ with $\int_{\gamma} \sigma = \infty$ for all $\gamma \in \Gamma_0$ and $\int_Z \sigma^Q < \infty$.  

We show $\rho + c \sigma$ is admissible for $\qmod(A,B)$ for any $c>0$. Let $\gamma$ be a path connecting $A$ and $B$. If $\gamma \in \Gamma_0$, then $\int_{\gamma} c \sigma = \infty$, so suppose $\gamma \notin \Gamma_0$.  Then, as $A$ and $B$ are open, $\gamma$ connects $A_{\lambda_n}$ and $B_{\lambda_n}$ for some $n$.  As $A_{\lambda_n} \subseteq A_{\lambda_k}$ for $n \leq k$, and likewise for $B$, we see $\int_{\gamma} \sigma_k \geq 1$ whenever $n \leq k$.  Hence, $\int_{\gamma} \rho_m \geq 1$ for all $m \geq n$, and so $\int_{\gamma} \rho \geq 1$.  Thus, $\rho + c \sigma$ is admissible for $\qmod(A,B)$.  Now, 
\begin{equation*}
\begin{split}
\norm{\rho + c \sigma}_Q^Q &\lesssim \norm{\rho}_Q^Q + \norm{c \sigma}_Q^Q \\
&\lesssim \qcap(A,B) + \norm{c \sigma}_Q^Q
\end{split}
\end{equation*}
and as $c$ may be taken arbitrarily small we have $\qmod(A,B) \lesssim \qcap(A,B)$.
\end{proof}

%%% SECOND THEOREM STARTS HERE

\begin{proof}[Proof of Theorem \ref{qcap < qmod}]
Recall that for this direction, in addition to working in a compact, connected Ahlfors $Q$-regular metric measure space $(Z,d,\mu)$ we also assume that $Z$ is a $Q$-Loewner space.  We first assume $A$ and $B$ are open sets with $\dist(A,B) > 0$.  For the continua case the proof will be the same except for one detail.  This is noted in the following proof and the required modifications will follow from Lemma \ref{continuum near 0}.

If $\qmod(A,B)=\infty$ there is nothing to prove so suppose $\qmod(A,B) < \infty$.  Let $\rho \colon Z \to [0,\infty]$ be $Q$-integrable and admissible for modulus.  Define $v \colon Z \to \R$ by
\begin{equation}
\label{v function def}
v(z) = \inf \left \{\int_{\gamma}{\rho}: \gamma \in \Gamma_z \right \}
\end{equation}
where $\Gamma_z$ is the set of all rectifiable paths with one endpoint in $A$ and the other endpoint equal to $z$.  It is clear that $\rho$ is an upper gradient for $v$; that is, given $y,z \in Z$, we have $|v(y)-v(z)| \leq \int_{\gamma} \rho$ whenever $\gamma$ is a rectifiable path connecting $y$ and $z$.  Hence, as $\rho$ is $Q$-integrable, it follows from \cite[Theorem 9.3.4]{HKST} that $v$ is measurable.  Set $u = \min\{v,1\}$. We see $\rho$ is an upper gradient for $u$ as well. Clearly $u(a)=0$ whenever $a\in A$ and, as $\rho$ is admissible for $Q$-modulus, we see $u(b) = 1$ whenever $b\in B$.   

By \cite[Theorem 5.12]{HK}, we note that $Z$ supports a $Q$-Poincar\'e inequality for continuous functions.  This is equivalent to a $Q$-Poincar\'e inequality for locally integrable functions by \cite[Theorem 8.4.1]{HKST}.  Following \cite[Theorem 12.3.9]{HKST}, first seen in \cite{KZ}, this promotes to a $p$-Poincar\'e inequality for functions which are integrable on balls for $p$ slightly smaller than $Q$.  

%Recall this means by setting $u_B = \dashint_B u $ there are $0 < \lambda \leq 1$ and $C > 1$ such that for all balls $B$ we have
%\begin{equation}
%\label{Poincare Inequality Eq}
%\dashint_{\lambda B} |u-u_{\lambda B} | \leq C (\diam B)\biggl (\dashint_B \rho^p\biggr )^{1/p}.
%\end{equation}
Now, consider the function $\tau \colon E\to \R$ by 
\begin{equation*}
\tau (e)= r(e_+)\biggl (\dashint_{K e_+} \rho^p \biggr )^{1/p} + r(e_-) \biggl (\dashint_{Ke_-} \rho^p \biggr)^{1/p}
\end{equation*}
with $K$ to be chosen and $e_+, e_-$ the balls representing the vertices of $e$.  As before, $Ke_+$ denotes the ball with the same center and as $e_+$ and with radius $K r(e_+)$.  We claim with appropriate $K$ that $\tau$ is admissible.  We note that for intersecting balls $B'$ and $B''$ with a constant $k \geq 1$ such that $B' \subseteq kB''$ and $B'' \subseteq kB'$, we have
\begin{equation*}
\begin{split}
|u_{B''} - u_{kB'}| &= \left | \frac{1}{|B''|} \int_{B''} u - \frac{1}{|B''|} \int_{B''} u_{kB'}  \right | \\
&\leq \frac{1}{|B''|} \int_{B''} |u - u_{kB'}| \\
&\leq \frac{1}{|B''|} \int_{kB'} |u - u_{kB'}| \\
&= \frac{|kB'|}{|B''|} \dashint_{kB'} |u - u_{kB'}| \\
& \lesssim \dashint_{kB'} |u - u_{kB'}|
\end{split}
\end{equation*}
and similarly $|u_{B'} - u_{kB'}| \lesssim \dashint_{kB'} |u_{kB'} - u|$.  Hence, by the triangle inequality,
\begin{equation*}
|u_{B'} - u_{B''}| \lesssim \dashint_{kB'} |u - u_{kB'}|
\end{equation*}
and so
\begin{equation*}
|u_{B'}- u_{B''}| \lesssim \dashint_{kB'} |u - u_{kB'}| + \dashint_{kB''} |u- u_{kB''}|.
\end{equation*}
In our graph, there is a uniform $k$ such that the above holds whenever $B'$ and $B''$ are vertices for a given edge. Thus, using the above notation where $e$ is an edge with $e_+$ and $e_-$ as the balls representing the vertices of $e$, we have
\begin{equation*}
\begin{split}
|u_{e_+}-u_{e_-}| &\lesssim \dashint_{k e_+}|u-u_{ke_+}| + \dashint_{k e_-}|u-u_{ke_-}| \\
&\lesssim r(e_+) \biggl (\dashint_{K e_+} \rho^p \biggr )^{1/p} + r(e_-) \biggl (\dashint_{K e_-} \rho^p \biggr )^{1/p} \\
&=\tau(e),
\end{split}
\end{equation*}
where $K = \lambda k$ arises from the Poincar\'e inequality (inequality \eqref{Poincare Inequality}).  Thus, if $\gamma$ is a path in the hyperbolic filling with limits in $A$ and $B$ we have, summing over the edges $e$ in $\gamma$,
\begin{equation}
\label{tau admissible telescope condition}
\sum_{\gamma}|u_{e_+}-u_{e_-}| \lesssim \sum_{\gamma}\tau(e).
\end{equation}
Now, $\gamma$ has boundary limits in $A$ and $B$.  We recall that if a sequence of vertices $B_n$ along a path approaches a limit $z \in Z = \partial_\infty X$, then the centers $p_n$ of the $B_n$ satisfy $p_n \to z$.  Thus, as $A$ and $B$ are open, for edges sufficiently close to $A$ we have $u_{e_+} = 0$ and for edges sufficiently close to $B$ we have $u_{e_+} = 1$.  We remark here that this fact is not true in the case that $A$ and $B$ are disjoint continua; this is where we will use Lemma \ref{continuum near 0}.  Hence, $1 \lesssim \sum_{\gamma} \tau(e)$ with constant independent of $\gamma$.  Thus, for suitable $c$ depending only on the constants $\lambda$ and $C$ from the Poincar\'e inequality and $k$ above, $c\tau$ is admissible.

It remains to compute $\norm{\tau}_{Q,\infty}$.  We follow a proof similar to \cite[Proposition 5.3]{BoS}.  We estimate the number of edges $e$ with an endpoint labelled as $v$ that belong to 
\begin{equation*}
V(\alpha) = \{v \in V : r(B_v) \biggl (\dashint_{KB_v}\rho^p \biggr )^{1/p} > \frac{\alpha}{2} \}
\end{equation*}
where $\alpha > 0$.  We will bound $\norm{\tau}_{Q,\infty}$ by bounding $\#V(\alpha)$.  Now, 
\begin{equation*}
\#V(\alpha) = \int_Z \biggl (\sum_{v \in V(\alpha)} \frac{1}{|B_v|}\chi_{B_v} \biggr )
\end{equation*}
and, by the geometric structure of our graph, we have, for $z \in Z$,
\begin{equation*}
\sum_{v \in V(\alpha)} \frac{1}{|B_v|} \chi_{B_v}(z) \lesssim \frac{1}{|B_z|}
\end{equation*}
where $B_z$ is the ball in $V(\alpha)$ of smallest radius that contains $z$ (such a ball exists for almost every $z$ as $\rho \in L^p$).

We note that for $v \in V(\alpha)$ and $z\in Kv$ we have
\begin{equation*}
\biggl (\dashint_{KB_v}\rho^p \biggr)^{1/p} \leq M(\rho^p)(z)^{1/p}
\end{equation*}
where $M$ denotes the (uncentered) Hardy-Littlewood maximal function (see for instance \cite[Chapter 2]{He}).  Thus,
\begin{equation*}
r(B_v)M(\rho^p)^{1/p}(z) > \frac{\alpha}{2}
\end{equation*}
for all $z \in KB_v$.  For such $v$, we see 
\begin{equation}
\label{alpha ineq}
\frac{1}{r(B_v)^Q} \leq \frac{2^Q M(\rho^p)(z)^{Q/p}}{\alpha^Q}.
\end{equation}
Hence, 
\begin{equation*}
\begin{split}
\# V(\alpha) &\lesssim \int_Z \biggl (\sum_{v \in V(\alpha)} \frac{1}{|B_v|}\chi_{B_v} \biggr ) \\
&\lesssim \int_Z \frac{1}{|B_z|} \\
&\lesssim \int_Z \frac{M(\rho^p)(z)^{Q/p}}{\alpha^Q} \\
&\lesssim \frac{1}{\alpha^Q} \int_Z (\rho^p )^{Q/p} \\
&= \frac{1}{\alpha^Q} \norm{\rho}_Q^Q
\end{split}
\end{equation*}
where we have used Ahlfors regularity with inequality (\ref{alpha ineq}) to bound $\frac{1}{|B_z|}$ and the fact that $Q/p > 1$ to bound the maximal function as an operator on $L^{Q/p}(Z)$.  Now, if an edge $e$ satisfies $\tau(e) > \alpha$ then at least one of its vertices must belong to $V(\alpha)$.  As $X$ has bounded valence, there is an $L>0$ such that each vertex can only occur as the boundary of at most $L$ edges.  Thus,
\begin{equation*}
\#\{e \in E : \tau(e) > \alpha\} \leq L \#V(\alpha) \lesssim \frac{L}{\alpha^Q} \norm{\rho}_Q^Q
\end{equation*}
From this we see $\norm{\tau}_{Q,\infty}^Q \lesssim \norm{\rho}_Q^Q$ and hence $\qcap(A,B) \lesssim \qmod(A,B)$ with a constant depending only on $Z$ and the hyperbolic filling.
\end{proof}

We now establish a lemma required to complete the proof of Theorem \ref{qcap < qmod} in the case of disjoint continua.  We continue to work with an admissible $\rho \in L^Q(Z)$.  For this lemma we need the following fact.

\begin{lemma}
\label{Cartan}
Let $\eta > 0$ and set
\begin{equation*}
E = E(\eta, \rho) = \left \{z \in Z:  \text{there exists }r > 0 \ \text{with } \int_{B(z,r)} \rho^Q \geq \eta r^Q \right \}.
\end{equation*}
Then $\Hq(E) \leq \frac{10^Q}{\eta} \int_Z \rho^Q$.
\end{lemma}

\begin{proof}
For each $z \in E$ there is a ball $B_z = B(z, r_z)$ for which $\int_{B_z} \rho^Q \geq \eta {r_z}^Q$.  As $Z$ is bounded it is clear we may assume the balls $B_z$ have uniformly bounded radius.  Hence, we may find a disjoint collection of these balls $B_{z_i}$ such that $E \subseteq \bigcup_i 5B_{z_i}$.  Thus, 
\begin{equation*}
\Hq(E) \leq \sum_i (10 r_{z_i})^Q = 10^Q \sum_i r_{z_i}^Q \leq \frac{10^Q}{\eta} \sum_i \int_{B_{z_i}} \rho^Q \leq \frac{10^Q}{\eta} \int_Z \rho^Q.
\qedhere
\end{equation*}
\end{proof}
We note that $Z$ here may be replaced by an appropriate smaller ambient space $Z'$ as long as we stipulate that we only consider radii $r$ for which $B(z,r) \subseteq Z'$.  Indeed, in our case we will apply this to $Z' = c_0B_v = B(z_0, c_0 R)$ with $c_0$ a constant that depends only on our path and the Loewner function. Thus, in this case the conclusion reads $\Hq(E) \leq \frac{10^Q}{\eta} \int_{c_0B_v} \rho^Q$.

% (all we need is for all scaled balls in the proof to be contained in $c_0B_v$; for this $c_0 = c_4^2 + 1$ will work where $c_4$ is defined below). 

We also use $u$ from the proof of Theorem \ref{qcap < qmod}.  Recall this means $u=\min(v,1)$ where $v$ is defined in equation (\ref{v function def}).  We also will carefully keep track of constants.  We will denote the Ahlfors regularity constants as $c_Q$ and $C_Q$.  That is, for balls $B'$ with radius $r$ bounded above by $\diam(Z)$ we have $c_Q r^Q \leq \mu(B') \leq C_Q r^Q$.
\begin{lemma}
\label{continuum near 0}
Consider a sequence of balls $B_v \to a \in A$ nontangentially.  Then $u_{B_v} = \dashint_{B_v} u \to 0$.  \end{lemma}
%%%%%%%

A rough outline of the proof is as follows: for a ball $B_v$ in the sequence we consider the set $M_v = \{u \geq \epsilon\} \cap B_v$.  For balls close enough to $A$, if $M_v$ is too large in $B_v$ we will use the Loewner condition to construct a path connecting $A$ to $M_v$ with short $\rho$-length.  When the $\rho$-length is less than $\epsilon$ this contradicts the definition of $u$.

\begin{proof}
Let $B_v  = B(z_0, R)$ be a ball in the sequence.  As our sequence approaches nontangentially, there is a constant $c_1 > 0$ depending only on our sequence such that $\dist(z_0, A) \leq c_1 R$.  We assume $B_v$ is close enough to $A$ that $ 4(1 + c_1) R < \diam(A)$.   Fix $\epsilon \in (0,1)$.  Set $M_v = \{u \geq \epsilon\} \cap B_v$.  We then note that by setting $\delta = \frac{\mu(M_v)}{\mu(B_v)}$ we have
\begin{equation}
\label{u delta ineq}
u_{B_v} \leq \frac{1}{\mu(B_v)} \bigl(\mu(M_v) + \epsilon\mu(B_v \setminus M_v) \bigr) = \delta + \epsilon(1-\delta).
\end{equation}
as $u \leq 1$.  We assume $\mu(M_v) > 0$ as our conclusion holds if $\mu(M_v) = 0$.

We now relate the measure of $M_v$ with its Hausdorff $Q$-content.  Indeed, if $B_i$ is a collection of balls covering $M_v$ with radii $s_i$, we have $\mu(M_v) \leq \sum_i \mu(B_i) \leq C_Q \sum_i s_i^Q$ and infimizing over all such collections yields $\frac{\mu(M_v)}{C_Q} \leq \Hq(M_v)$.

Now by applying Lemma \ref{Cartan} we see that the Hausdorff $Q$-content of the set 
\begin{equation*}
\begin{split}
E&=E(\eta, v)\\ &= \left \{z \in M_v: \text{there exists } r>0 \text{ such that } B(z,r) \subseteq c_0B_v \text{ and } \int_{B(z,r)} \rho^Q \geq \eta r^Q \right \}
\end{split}
\end{equation*}
satisfies $\Hq(E) \leq \frac{10^Q}{\eta} \int_{c_0B_v} \rho^Q$.

If $\int_{c_0 B_v} \rho^Q = 0$ then we set $\eta = 1$ and otherwise we set 
\begin{equation}
\label{eta def}
\eta = \frac{2 C_Q 10^Q \int_{c_0B_v} \rho^Q}{\mu(M_v)}.
\end{equation}
Hence, 
\begin{equation*}
\Hq(E) \leq \frac{10^Q}{\eta} \int_{c_0B_v} \rho^Q = \frac{\mu(M_v)}{2 C_Q} < \frac{\mu(M_v)}{C_Q} \leq \Hq(M_v).
\end{equation*}
It follows from $\Hq(E) < \Hq(M_v)$ that there is an $x \in M_v \setminus E$ such that for all $r>0$ with $B(x,r) \subseteq c_0B_v$ we have $\int_{B(x,r)} \rho^Q < \eta r^Q$.

Recall $\dist(z_0, A) \leq c_1 R$ and so $\dist(x,A) \leq (1+c_1) R$.  We set $B_0 = B(x, c_2 R)$ and $B_j = 2^{-j} B_0$ where $c_2 = 2(1+c_1)$.  As $4(1 + c_1) R < \diam(A)$ there is a subcontinua $E_0 \subseteq A$ that satisfies $E_0 \subseteq B_0$ and $\diam(E_0) \geq \frac{c_1}{2} R$. Indeed, by possibly removing some of $E_0$ we may also assume $E_0 \subseteq B_0 \setminus B_1$.  As $Z$ is a complete and doubling metric measure space that supports a Poincar\`e inequality, $Z$ is quasiconvex (see \cite[Theorem 8.3.2]{HKST}) and hence rectifiably path connected. Thus, there is a rectifiable path $\beta$ connecting $x$ to $E_0$, say $\beta:[0,1] \to Z$ with $\beta(0) = x$ and $\beta(1) \in E_0$.  We define continua $E_j$ as follows: given $j >0$, let $t_j^-$ denote the first time after which $\beta$ does not return to $B_{2j+1}$ and let $t_j^+$ denote the first time $\beta$ leaves $B_{2j}$.  Then set $E_j = \beta([t_j^-, t_j^+])$.  We note that for each $j$ we have $E_j \subseteq B_{2j}$ and $\diam(E_j) \geq \frac{1}{10}\diam(B_{2j})$.  Hence, it follows that 
\begin{equation*}
\begin{split}
\Delta(E_j, E_{j+1}) &= \frac{\dist(E_j, E_{j+1})}{\min(\diam(E_j),\diam(E_{j+1}))} \\
&\leq \frac{\diam(B_{2j})}{\frac{1}{10} \diam(B_{2j+2})} \\
&\leq 10 \frac{2 (2^{-(2j)}c_2 R)}{2(2^{-(2j+2)}c_2R)} \\
&= 40.
\end{split}
\end{equation*}
As we are in a $Q$-Loewner space, this means $\qmod(E_j,E_{j+1}) \geq \varphi(40)$, where $\varphi$ is the $Q$-Loewner function associated to $Z$.  As $Z$ is Ahlfors $Q$-regular, it follows from \cite[Proposition 5.3.9]{HKST} that there is a constant  $c_3 > 0$ such that 
\begin{equation*}
\qmod(\overline{B}(x_0,s), Z \setminus B(x_0, S)) \leq c_3 \biggl(\log \frac{S}{s} \biggr)^{1-Q}
\end{equation*}
when $0 < 2s < S$.  In particular, we can find a constant $c_4 > 2$ such that if $\Gamma^*(E_j, E_{j+1})$ is the path family of rectifiable paths connecting $E_j$ to $E_{j+1}$ that leaves $c_4 B_{2j}$, we have $\qmod(\Gamma^*(E_j, E_{j+1})) \leq \frac{\varphi(40)}{2}$.  %Here we note we may define $c_0$ as $2 c_4 c_2$.  
Thus, the $Q$-modulus of the family of rectifiable paths connecting $E_j$ to $E_{j+1}$ which stay inside $c_4 B_{2j}$ is at least $\frac{\varphi(40)}{2}$.  In particular, this means for each $j$ that one can find a path $\alpha_j$ that connects $E_j$ and $E_{j+1}$, stays inside $c_4B_{2j}$, and satisfies 
\begin{equation*}
\int_{\alpha_j} \rho \leq \biggl( \frac{4 \int_{c_4 B_{2j}} \rho^Q}{\varphi(40)} \biggr)^{1/Q}.
\end{equation*}
Here if $\int_{c_0 B_v} \rho^Q = 0$ then instead for each $\nu>0$ we can find $\alpha_j(\nu)$ such that $\int_{\alpha_j} \rho \leq \nu$ and the following argument works by choosing values of $\nu$ that are sufficiently small.
Recall we also use $\alpha_j$ to denote the image of $\alpha_j$.  Hence, each $\alpha_j$ is a continuum, 
\begin{equation*}
\dist(\alpha_j,\alpha_{j+1}) \leq \diam(E_{j+1}) \leq 2(2^{-(2j+2)})c_2 R,
\end{equation*}
and 
\begin{equation*}
\diam(\alpha_j) \geq \dist(E_j, E_{j+1}) \geq \frac{1}{4} \diam(B_{2j+1}) \geq \frac{1}{2} 2^{-(2j+1)} c_2 R.
\end{equation*}
Thus, we see
\begin{equation*}
\Delta(\alpha_j, \alpha_{j+1}) \leq \frac{2(2^{-(2j+2)})c_2 R}{\frac{1}{2} (2^{-(2j+3)}) c_2 R} = 8 < 40.
\end{equation*}
Hence we may perform the same procedure as above to find paths $\beta_j$ connecting $\alpha_j$ to $\alpha_{j+1}$ with 
\begin{equation*}
\int_{\beta_j} \rho \leq \biggl( \frac{4 \int_{c_4^2 B_{2j}} \rho^Q}{\varphi(40)} \biggr)^{1/Q}
\end{equation*}
where we note the $c_4^2$ arises as $\alpha_j \subseteq c_4 B_{2j}$.
From the way these paths were constructed it is clear we can extract a rectifiable path $\gamma$ from $\bigcup_j (\alpha_j \cup \beta_j)$ connecting $M_v$ to $A$ such that
\begin{equation*}
\int_\gamma \rho \leq  \sum_j \biggl( \frac{4 \int_{c_4 B_{2j}} \rho^Q}{\varphi(40)} \biggr)^{1/Q} + \biggl( \frac{4 \int_{c_4^2 B_{2j}} \rho^Q}{\varphi(40)} \biggr)^{1/Q}.
\end{equation*}
As $x \notin E(\eta)$, we know that $\int_{c B_{2j}} \rho^Q \leq \eta (c 2^{-(2j)} c_2 R)^Q$ for each $c >0 $ with $c B_0 \subseteq c_0 B_v$.  We note here that this is the requirement on $c_0$, namely that $B(x, c_4^2 c_2 R) \subseteq c_0 B_v$.  Hence, 
\begin{equation*}
\int_{\gamma}\rho \leq \sum_j \biggl(\frac{4}{\varphi(40)}  \eta (c_4 2^{-(2j)} c_2 R)^Q \biggr)^{1/Q} + \biggl(\frac{4}{\varphi(40)}  \eta (c_4^2 2^{-(2j)} c_2 R)^Q \biggr)^{1/Q} \lesssim \eta^{1/Q} R
\end{equation*}
with constant only depending on $c_4$ and $\varphi(40)$ (in the case that $\int_{c_0 B_v} \rho^Q = 0$ we can instead make $\int_{\gamma}\rho$ as small as we like). Recall the definition of $\eta$ given by (\ref{eta def}) which gives
\begin{equation*}
\int_{\gamma}\rho \lesssim \biggl( \int_{c_0 B_v} \rho^Q \biggr)^{1/Q} \biggl(\frac{1}{\mu(M_v)^{1/Q}}\biggr)R.
\end{equation*}
Now, $\delta = \frac{\mu(M_v)}{\mu(B_v)}$ and $\mu(B_v) \geq c_Q R^Q$, from which we conclude 
\begin{equation*}
\int_{\gamma} \rho \lesssim \biggl( \int_{c_0 B_v} \rho^Q \biggr)^{1/Q} \delta^{-1/Q}.
\end{equation*}
Now, as $x \in M_v$ and $\gamma$ is a path connecting $A$ to $x$, we must have $\int_{\gamma} \rho \geq \epsilon$.  Thus,
\begin{equation*}
\epsilon \delta ^{1/Q} \leq \biggl( \int_{c_0 B_v} \rho^Q \biggr)^{1/Q}.
\end{equation*}
For fixed $\epsilon>0$ the right hand side tends to $0$ as $v \to a \in A$. Thus, we must have $\delta \to 0$ as $v \to a$.  Hence, from inequality (\ref{u delta ineq}) we conclude $u_{B_v} \to 0$ as $v \to a$.
\end{proof}

The above argument can be adapted to show that as $B_v \to B$ we have $u_{B_v} \to 1$.  To do this, we would instead use $M_v = \{ u \leq 1 - \epsilon \} \cap B_v$ and argue as above that if $\delta = \frac{\mu(M_v)}{\mu(B_v)}$ was large then there would exist a path from $M$ to $B$ with short $\rho$-length.  From the definition of $u$ this would produce a path $\gamma$ connecting $A$ to $B$ with total length less than $1$, contradicting the admissibility of $\rho$.

This completes the proof of the continua case by bounding below the quantity in inequality (\ref{tau admissible telescope condition}).
%%%%%%%%%

Lastly we prove Theorem \ref{QS inv qcap}, the quasisymmetric invariance property for $\pcap$.

\begin{proof}[Proof of Theorem \ref{QS inv qcap}]
Recall $Z$ and $W$ are compact, connected, Ahlfors regular metric spaces and $\varphi: Z \to W$ is an $\eta$-quasisymmetry.  Let $X=(V_X, E_X)$ and $Y=(V_Y, E_Y)$ be corresponding hyperbolic fillings.  The proof for open sets and continua is the same, so let $A$ and $B$ either be open sets with $\dist(A,B) > 0$ or disjoint continua.  As $\varphi^{-1}$ is an $\eta'$-quasisymmetry with $\eta'$ depending on $\eta$, it suffices to show $\pcap(\varphi(A),\varphi(B)) \lesssim \pcap(A,B)$.  Let $G \colon Y \to X$ be the quasi-isometry induced by $\varphi^{-1}$ as in Lemma \ref{QI induces QS}.  We note that $G$ maps vertices to vertices.  Let $D>0$ be such that for all adjacent vertices $y,w \in Y$, $|G(y) - G(w)| \leq D$.   

Let $\tau \geq 0$ be admissible for $\pcap(A,B)$.  We construct $\sigma$ on $E_Y$ as follows: given an edge $e'\in E_Y$ with vertices $e'_+$ and $e'_-$, we set
\begin{equation*}
\sigma(e') = \sum_{|x - G(e'_+)| \leq D}\biggl( \sum_{e \sim x} \tau(e) \biggr) + \sum_{|x - G(e'_-)| \leq D} \biggl( \sum_{e \sim x} \tau(e) \biggr)
\end{equation*}	
where $e \sim x$ means $e$ is an edge that has the vertex $x$ as an endpoint.

We show that $\sigma$ is admissible for $\pcap(\varphi(A),\varphi(B))$.  Indeed, if $\gamma$ is a path in $Y$ with limits in $\varphi(A)$ and $\varphi(B)$, then we construct a path in $X$ with limits in $A$ and $B$ that serves as a suitable image of $\gamma$.  Each vertex $y \in \gamma$ corresponds to a point $G(y) \in X$.  By our choice of $D$, if two vertices $y$ and $y'$ are connected by an edge in $\gamma$, then $|G(y) - G(y')| \leq D$.  We choose a path connecting $G(y)$ and $G(y')$ that stays in the ball of radius $D$ centered at $G(y)$.  By doing this for all connected vertices in $\gamma$, we produce a path $\gamma_X$ in $X$ with limits in $A$ and $B$.  Now, by construction,
\begin{equation*}
\sum_{e' \in \gamma} \sigma(e') \geq \sum_{y \in \gamma} \biggl( \sum_{|x - G(y)| \leq D} \biggl( \sum_{e \sim x} \tau(e) \biggr) \biggr) \geq \sum_{e \in \gamma_X} \tau(e) \geq 1,
\end{equation*}
where we have viewed $\gamma$ as both a sequence of vertices and a sequence of edges.  The last inequality follows as $\tau$ was assumed admissible for $\pcap(A,B)$.

It remains to show $\norm{\sigma}_{p,\infty} \lesssim \norm{\tau}_{p,\infty}$.  For this we use Lemma \ref{BoS Lemma}.  Our set $J \subseteq E_Y \times E_X$ consists of pairs $(e',e)$ for which $e$ appears as a summand in the definition of $\sigma(e')$.  Following the notation from Lemma \ref{BoS Lemma}, for $e' \in E_Y$ the set $J_{e'}$ is the set of edges $e$ that appear as a summand in the definition of $\sigma(e')$.  The cardinality $|J_{e'}|$ is bounded independent of $e'$ as $X$ has bounded valence.  Similarly, for $e \in E_X$ the set $J^e$ is the set of $e'$ for which $e$ contributes to the sum in the definition of $\sigma(e')$.  We show that a given $e$ can only contribute to a bounded number of such $\sigma(e')$.  Indeed, if $(e',e) \in J$, then one of the vertices of $e$ must lie in the $D$ radius ball around the image under $G$ of one of the vertices of $e'$.  As $G$ is a quasi-isometry, we see there is a constant $C > 0$ such that if $y, y' \in V_Y$ and $|y-y'| > C$, then $|G(y) - G(y')| > 2D + 1$.  Hence, for edges $e''$ far enough away from $e'$ in $Y$, we must have $e'' \notin J^e$.  As $Y$ has bounded valence by Lemma \ref{bounded valence}, we see $|J^e|$ is bounded independent of $e$.

We set $s_{e'}=\sigma(e')$ and $t_e=\tau(e)$.  Inequality (\ref{BoS Lemma sum condition}) in Lemma \ref{BoS Lemma} then follows from the definition of $\sigma$.  Hence, we conclude $\norm{\sigma}_{p,\infty} \lesssim \norm{\tau}_{p,\infty}$.
\end{proof}

We now show the positivity of $\pcap$: whenever $A, B \subseteq Z$ are open sets with $\dist(A,B) > 0$, we have $\pcap(A,B) > 0$.  Here only $p \geq 1$ is assumed.  Recall we work with a fixed hyperbolic filling $X = (V,E)$ with parameter $s > 1$.  To prove this result, Proposition \ref{positivity pcap}, we first detail a construction.  For this we need the following lemma. 

\begin{lemma}
\label{Vertex splitting lemma}
There exists a constant $M>0$ such that whenever $v \in V$ is a vertex in $X$, there exist two vertices $v_1, v_2$ with levels $\ell(v_j) = \ell(v) + M$, $2 B_{v_1} \cap 2 B_{v_2} = \emptyset$, and $B_{v_1}, B_{v_2} \subseteq B_v$.
\end{lemma}

\begin{proof}
We write $B_v = B(z_v, r_v)$.  We recall $Z$ is Ahlfors $Q$-regular, so there exist constants $c, C > 0$ such that for all $z \in Z$ and $r \in (0, \diam(Z))$, we have $c r^Q \leq \mu (B(z, r)) \leq C r^Q$.  Fix small $k > 0$ such that $(\frac{3/4}{k})^Q > \frac{C}{c}$.  Then 
\begin{equation*}
\mu \left(B(z_v, (3/4) r_v) \setminus B(z_v, k r_v) \right) \geq c (3/4)^Q {r_v}^Q - C k^Q {r_v}^Q
\end{equation*}
and by our choice of $k$ we see 
\begin{equation*}
c (3/4)^Q {r_v}^Q - C k^Q {r_v}^Q > 0.
\end{equation*}
Hence, there is a point $z_1 \in B(z_v, (3/4) r_v) \setminus B(z_v, k r_v)$.  There is also a point $z_2 \in B(z_v, \frac{k}{2} r_v)$.  For example, let $z_2 = z_v$.

Let $M$ be such that $s^{-M} < \frac{k}{16}$, where we recall $s>1$ was a parameter in the construction of the hyperbolic filling.  Then as the balls of level $\ell(v) + M$ cover $Z$, there must be balls $B_j$ of radius $2 s^{-(\ell(v) + M)}$ with $z_j \in B_j$.  Now, $\dist(z_1, z_2) \geq \frac{k}{2} r_v$ by construction.  As $r_v = 2 s^{-\ell(v)}$, this is $\dist(z_1, z_2) \geq k s^{-\ell(v)}$.   Hence, the sum of the diameters of the balls $2 B_j$ is bounded by 
\begin{equation*}
16 s^{-(\ell(v) + M)} < 16 s^{-(\ell(v))} \frac{k}{16} = ks^{-(\ell(v))} \leq \dist(z_1, z_2)
\end{equation*}
and so $2B_1 \cap 2B_2 = \emptyset$.
\end{proof}

We now construct our path structure.  Let $v \in V$.  Let $G = \{0, 1\}^*$ be the set of finite sequences of elements of $\{0,1\}$.  For an element $g \in G$, we let $g0$ and $g1$ denote the concatenations of the symbols $0$ and $1$ to the right hand side of $g$.  We associate elements of $G$ to vertices in $V$ inductively as follows: let $\emptyset$ correspond to $v$ and, given an element $g \in G$ with corresponding vertex $v_g$, apply Lemma \ref{Vertex splitting lemma} to $B_{v_g}$ to obtain $B_{v_{g0}}$ and $B_{v_{g1}}$ which correspond to $g0$ and $g1$.  We also choose $\gamma_{g0}$ (and likewise $\gamma_{g1}$) to be an edge path connecting $v_g$ to $v_{g0}$ with length $M$.  To construct such a path, one can choose a point in $B_{v_{g0}}$ and select vertices corresponding to balls containing that point with levels between those of $v_g$ and $v_{g0}$.  To form $\gamma_{g0}$ one then uses the edges between these vertices.

\begin{remark}
\label{2B remark}
We note from the fact that $2B_1 \cap 2B_2 = \emptyset$ in Lemma \ref{Vertex splitting lemma} that,  if $g, h \in G^N$, then $\gamma_g \cap \gamma_h \neq \emptyset$ can only happen if the first $N-1$ entries of $g$ match those of $h$.
\end{remark}

\begin{defn}
Given a vertex $v \in V$, we call a path structure as constructed above a binary path structure and denote it $T_{v}$.
\end{defn}

We call an edge path $\gamma = (e_k)$ ascending if the levels of the endpoints of successive edges is strictly increasing.  That is, if $(v_k)$ is the sequence of vertices that $\gamma$ travels through, then $\ell(v_{k+1}) = \ell(v_k) + 1$ for all $k$.  We consider functions $\tau: E \to [0, \infty]$ which are admissible on ascending edge paths originating from $v$.  That is, for such paths $\gamma$ we require $\sum_\gamma \tau(e_k) \geq 1$.  We claim such functions cannot have too small weak $\ell^p$-norm.

\begin{lemma}
Let $\tau: E \to [0, \infty]$ and $v \in V$.  Then, there is a constant $S(p) < \infty$ and an ascending edge path $\gamma = (e_k)_{k=0}^\infty$ from $v$ to $\overline{B_v} \subseteq Z$ with $\tau$-length bounded above by $\norm{\tau}_{p,\infty} S(p)$.
\end{lemma}

Functions that are admissible on ascending edge paths must give at least $\tau$-length 1 to such paths, so for these functions we have $\norm{\tau}_{p,\infty} \geq 1/S(p)$.

\begin{proof}
We may assume $\norm{\tau}_{p,\infty} < \infty$.  Let $T_v$ be a binary path structure originating from $v$.  For $N \in \N$, let $G^N = \{0,1\}^N$ be the set of finite strings of elements of $\{0,1\}$ of length $N$.  For $g \in G^N$ and $k<N$, let $g_k$ denote the element of $G^k$ which matches the first $k$ entries of $g$.  We also set $\gamma'_g = \bigcup_k \gamma_{g_k}$ to be the ascending edge path formed by concatenating $\gamma_{g_1},\dots,\gamma_{g}$.  The average $\tau$-length of the paths $\gamma'_g$, where $g \in G^N$, is given by
\begin{equation}
\label{path average}
\begin{split}
\frac{1}{2^N} \sum_{g \in G^N} \sum_{e \in \gamma'_g} \tau(e) &= \frac{1}{2^N} \sum_{g \in G^N} \sum_{k=1}^N  \biggl( \sum_{e \in \gamma_{g_k}} \tau(e) \biggr) \\
&= \frac{1}{2^N} \sum_{k=1}^N 2^{N-k} \sum_{h \in G^k} \biggl( \sum_{e \in \gamma_h}  \tau(e) \biggr).
\end{split}
\end{equation}
We bound this average above using $\norm{\tau}_{p,\infty}$.  For ease of notation, let $a = \norm{\tau}_{p,\infty}$.  By definition, we have
\begin{equation*}
\# \{e : \tau(e) > \lambda \} \leq \frac{a^p}{\lambda^p}.
\end{equation*}
Hence, for $j \in \N$ the function $\tau$ can take values $\tau(e) \geq \frac{a}{j^{(1/p)}}$ for at most $j$ edges $e$.  From Remark \ref{2B remark} for fixed $k > 1$ there are at least $(2^{(k-1)} - 1) M$ distinct edges contributing to the right hand side of equation (\ref{path average}) belonging to paths $\gamma_f$ with $f \in G^\ell$ where $\ell < k$.  With the weighting factors $2^{N-k}$ it follows that the average increases the more $\tau$-mass is located on paths with smaller associated $k$ values.  Using these observations we conclude that to bound equation (\ref{path average}) above, for $k>1$ and $h \in G^k$ we bound the sum $\sum_{e \in \gamma_h} \tau(e)$ above by 
\begin{equation*}
M \frac{a}{((2^{(k-1)} - 1)M)^{(1/p)}}.
\end{equation*}
From this, we obtain the upper bound
\begin{equation*}
\frac{1}{2} \sum_{h \in G^1} \biggl( \sum_{e \in \gamma_h}  \tau(e) \biggr)  + \sum_{k=2}^N 2^{-k} \sum_{h \in G^k} M \frac{a}{((2^{(k-1)} - 1)M)^{(1/p)}}.
\end{equation*}
We bound the first term by noting that $\tau(e) \leq a$ for all $e \in E$.  As $G^k$ has $2^k$ elements, our bound becomes
\begin{equation*}
Ma + \sum_{k=2}^N M \frac{a}{((2^{(k-1)} - 1)M)^{(1/p)}}.
\end{equation*}
Setting
\begin{equation*}
S(p) = M + \sum_{k=2}^\infty M \frac{1}{((2^{(k-1)} - 1)M)^{(1/p)}} < \infty
\end{equation*}
we thus bound (\ref{path average}) above by $a S(p) = \norm{\tau}_{p,\infty} S(p)$.  As this is a bound on the average $\tau$-length of the paths $\gamma'_g$, where $g \in G^N$, we conclude that for each $N$ there is a $g_N \in G^N$ such that $\gamma'_{g_N}$ has $\tau$-length bounded above by $\norm{\tau}_{p,\infty} S(p)$.

To complete the proof we construct a path $\gamma$ from the paths $\gamma'_{g_N}$.  For each $N$ we have $g_N \in \{0, 1\}^N$.  Hence, there is a subsequence of the $g_N$ that has the property that all strings in this subsequence have the same first element.  From this subsequence, we may extract another subsequence that consists of strings that all have the same first two elements.  Continuing in this manner and then diagonalizing produces an infinite sequence $h \in \{0,1\}^\N$ for which, for infinitely many $k \in \N$, the first $k$ elements of $h$ match $g_k$.  It is clear how to associate $h$ to an ascending path $\gamma$ which, from the bounds on the $\tau$-lengths of the paths $\gamma'_{g_k}$, satisfies the conclusion of the lemma.
\end{proof}

We perform a similar analysis on a path of edges of length $L$.  As above, if $\norm{\tau}_{p,\infty} = b$, then $\tau$ can take values $\tau(e) \geq \frac{b}{j^{(1/p)}}$ at most $j$ times.  Thus, the maximum $\tau$-length our line can have is $\sum_{k=1}^L \frac{b} {k^{(1/p)}}$.  This is bounded above (for $p > 1$) by $b(\int_0^L \frac{1}{x^{(1/p)}} dx)$, which is $\frac{b L^{1 - (1/p)}}{1 - (1/p)}$.  For $p=1$ this is bounded above by $b(1+\log(L))$.

We are now ready to prove positivity.

\begin{prop}
\label{positivity pcap}
Let $p \geq 1$ and let $A,B \subseteq Z$ be open sets with $\dist(A,B) > 0$.  Then, $\pcap(A,B) > 0$.  
\end{prop}

\begin{proof}
We first assume $p > 1$.  Let $v, w$ be vertices such that $\overline{B_v} \subseteq A$ and $\overline{B_w} \subseteq B$ and such that $\ell(v) = \ell(w)$.  We connect $v$ and $w$ by an edge path $\gamma$ contained in $\{x \in X: |x-O| \leq \ell(v)\}$, where $O$ is the unique vertex with $\ell(O) = 0$.  We let $L$ denote the length of $\gamma$.  Recall $T_v$ and $T_w$ denote binary path structures originating from $v$ and $w$.  Now, if $\tau$ is admissible for $\pcap(A,B)$, then $\tau$ is admissible on paths contained in $T_v \cup T_w \cup \gamma$.  Let $\tau_A, \tau_B,\text{ and } \tau_\gamma$ denote the restrictions of $\tau$ to $T_v, T_w, \text{ and } \gamma$.  We then must have
\begin{equation*}
\norm{\tau_A}_{p,\infty} S(p) + \norm{\tau_B}_{p,\infty} S(p) + \norm{\tau_\gamma}_{p,\infty} \frac{L^{1-(1/p)}}{1-(1/p)} \geq 1.
\end{equation*}
As $\norm{\tau}_{p,\infty}$ is larger than the norm of each of these restrictions, we see
\begin{equation*}
\norm{\tau}_{p,\infty} S(p) + \norm{\tau}_{p,\infty} S(p) + \norm{\tau}_{p,\infty} \frac{L^{1-(1/p)}}{1-(1/p)} \geq 1.
\end{equation*}
Hence,
\begin{equation*}
\norm{\tau}_{p,\infty} \geq \frac{1}{2S(p) + \frac{L^{1-(1/p)}}{1-(1/p)}}.
\end{equation*}
The above analysis also applies to the $p=1$ case with the appropriate bound modification.  In this case, this bound becomes
\begin{equation*}
\norm{\tau}_{1,\infty} \geq \frac{1}{2S(1) + (1+\log(L))}.
\end{equation*}
Both cases thus yield a lower bound on $\norm{\tau}_{p,\infty}$ for admissible $\tau$, as desired.
\end{proof}

\section{Ahlfors regular conformal dimension}
\label{Cdmin}

We define a critical exponent relating to $\wcap$ which is motivated by a critical exponent defined in \cite{BK}.  This is not the first attempt to define meaningful critical exponents using hyperbolic fillings.  For example, see \cite{CP}.

\begin{defn}
Let 
\begin{equation*}
Q_w = \inf\{p:\pcap(A,B) < \infty \text{ for all open } A \text{ and }B \text{ with } \dist(A,B) >0\}.
\end{equation*}
\end{defn}
We write $\ARCdim$ for the Ahlfors regular conformal dimension of $Z$.  That is, if 
\begin{equation*}
\mathscr{G} = \{\theta: \theta \text{ is a metric on } Z \text{ with } (Z,\theta) \sim_{qs} (Z,d)\},
\end{equation*}
then $\ARCdim = \inf  \dim_H (Z,\theta)$ where the infimum is taken over all $\theta \in \mathscr{G}$ such that $(Z,\theta)$ is Ahlfors regular.  Here $(Z,\theta) \sim_{qs} (Z,d)$ means that the identity map is a quasisymmetry.

\begin{lemma} 
We have $Q_w \leq \ARCdim$.  
\end{lemma}
\begin{proof}
From quasisymmetric invariance, it suffices to show for any $p > \ARCdim$ and any open sets $A$ and $B$ with $\dist(A,B) > 0$ we have $\wcap_p(A,B) < \infty$.  Fix such parameters and let $\theta$ be an Ahlfors regular metric on $Z$ such that $\dim_H (Z,\theta) \leq p$.  Let $X=(V,E)$ be a hyperbolic filling for $(Z,\theta)$ with parameter $s>1$ and consider $f \colon V \to \R$ defined by 
\begin{equation*}
f(v) = \frac{4r(B_v)}{\dist(A,B)}.
\end{equation*}
It follows from $\dim_H (Z,\theta) \leq p$ that the number of vertices on level $n$ is bounded above by $C s^{np}$ for some $C>0$ and so $f \in \ell^{p,\infty}(V)$.  Define $\tau \colon E \to \R$ by $\tau(e) = f(e_+) + f(e_-)$.  Hence, $\norm{\tau}_{p,\infty} \lesssim \norm{f}_{p,\infty} < \infty$ follows from Lemma \ref{BoS Lemma}.  Thus, we need only show $\tau$ is admissible.  

Now, if $\gamma$ is any finite chain of vertices, say $\{v_0, \dots, v_N\}$ (where $v_k$ is connected to $v_{k+1}$ for all $k$), then 
\begin{equation*}
\sum_\gamma(\tau(e)) \geq \sum_k f(v_k) \geq \frac{2\dist(B_{v_0}, B_{v_N})}{\dist(A,B)}.
\end{equation*}
Thus, as an infinite $\gamma$ with nontangential limits in $A$ and $B$ has a finite subpath $\gamma_0 = \{v_0, \dots, v_N\}$ with $B_{v_0} \cap A \neq \emptyset$ and $B_{v_N} \cap B \neq \emptyset$ and small enough radii so that $2\dist(B_{v_0}, B_{v_N}) \geq \dist(A,B)$, we have
\begin{equation*}
\sum_\gamma \tau(e) \geq \sum_{\gamma_0} \tau(e) \geq \frac{2\dist(B_{v_0}, B_{v_N})}{\dist(A,B)} \geq 1
\end{equation*}
so $\tau$ is admissible.  
\end{proof}

In view of this inequality, one question is how does $Q_w$  relate to  $\ARCdim$?  If $Q_w \neq \ARCdim$, then can one define a similar critical exponent that is $\ARCdim$?  The following result shows that for some metric spaces we do have equality.

\begin{lemma}
Let $(Z,d,\mu)$ be a compact, connected metric measure space that is Ahlfors $Q$-regular, $Q > 1$, and such that there exists $1 \leq p \leq Q$ and a family of paths $\Gamma$ with $\modp(\Gamma) > 0$.  Then there exist open balls $A$ and $B$ with $\dist(A,B) > 0$ such that for all $q < Q$ we have $\capq(A,B) = \infty$.
\end{lemma}

In such metric spaces \cite[Proposition 4.1.8]{MT} shows $\ARCdim(Z,d) = Q$ and so we have $Q_w = \ARCdim$.

\begin{proof}
From \cite[Proposition 4.1.6 (vii)]{MT} it follows that $\qmod(\Gamma) > 0$.  By potentially taking subpaths, we may assume every path in $\Gamma$ has distinct end points (i.e. no paths in $\Gamma$ are loops).  By writing $\Gamma = \cup(\Gamma_m)$ where $\Gamma_m = \{\gamma \in \Gamma: \ell(\gamma) > \frac{1}{m}\}$, we may assume the paths in $\Gamma$ have lengths uniformly bounded from below (we need $Q>1$ for this, see \cite[Proposition 4.1.6 (iv)]{MT}).  By covering $Z$ with a finite number balls of small enough radius and writing $\Gamma$ as the union of paths connecting two disjoint balls with positive separation, we may assume $\Gamma$ connects two open balls $A$ and $B$ with $\dist(A,B)>0$.  Refining this slightly allows us to assume $\Gamma$ connects $A_\lambda$ and $B_\lambda$ for some $\lambda>0$ as in the proof of Theorem  \ref{qmod < qcap}.

Now suppose $\tau: E \to \R$ is admissible for $\capq(A,B)$ and satisfies $\norm{\tau}_{q,\infty} < \infty$, where $q < Q$.  Define $f: V \to \R$ by $f(v) = \sum_{e \sim v} \tau(e)$ and, as in the proof of Theorem \ref{qmod < qcap}, set \begin{equation*}
u_n = \sum_{v \in V_n} \frac{f(v)}{r(B_v)}\chi_{2B_v}.
\end{equation*}
Then, as before, for large enough $n$ the function $2u_n$ is admissible for $\Gamma$.  Our estimate for $\norm{u_n}_Q^Q$ is also the same as in inequality (\ref{u_n bound}):
\begin{equation*}
\begin{split}
\norm{u_n}_Q^Q \lesssim \sum_{v \in V_n} f(v)^Q.
\end{split}
\end{equation*}
As $f \in \ell^{q,\infty}(V)$, it follows that $f \in \ell^Q(V)$ and so $\sum_{v \in V_n} f(v)^Q \to 0$ as $n \to \infty$.  This shows $\qmod(\Gamma) = 0$, a contradiction.  Hence, no such admissible function exists and $\capq(A,B) = \infty$.
\end{proof}

%SECTION WEAK COVERING Q CAP

\section{Weak covering capacity}
\label{Sec wcqcap}

Here we state and prove some basic properties and the main theorems involving $\wcpcap$.  Recall we work on a compact, connected, Ahlfors $Q$-regular metric measure space $(Z,d, \mu)$ with $Q > 1$ and with a fixed hyperbolic filling $X=(V_X, E_X)$ of $(Z,d,\mu)$ with scaling parameter $s > 1$.  We first prove that if $p \geq Q$, then $\wcpcap$ is supported on rectifiable paths.

\begin{lemma}
\label{wcpcap nonrectifiable}
Let $p \geq Q$ and let $\Gamma_\infty$ be the set of all infinite-length paths $\gamma \colon [0,1] \to Z$.   Then $\wcpcap(\Gamma_\infty) = 0$.
\end{lemma}

\begin{proof}
We show that for any $\epsilon > 0$, the functions $\tau_\epsilon(v) = r(B_v) \epsilon$ are admissible for $\Gamma_\infty$ and that $\norm{\tau_\epsilon}_{p,\infty} \to 0$ as $\epsilon \to 0$.  From this it follows that $\wcpcap(\Gamma_\infty) = 0$.  Fix $\epsilon > 0$.  Let $\gamma \in \Gamma_\infty$.  Let $t_0 < t_1 < \dots < t_m$ be a partition of $[0,1]$ such that the points $\gamma(t_k)$ are distinct and $\sum_k d(\gamma(t_{k-1}),\gamma(t_k)) >  $ 4/$\epsilon$. Set $\gamma_k = \gamma|_{[t_{k-1},t_k]}$. Let $N$ be such that $400s^{-N} \leq \min_k d(\gamma(t_{k-1}),\gamma(t_k))$ and $m 2 \epsilon s^{-N} < 1$.  Let $\{S_j\}$ be an expanding sequence of covers.  As $\{S_j\}$ is expanding, for large enough $j$ we see that the balls in $S_j$ all have radius bounded above by $2s^{-N}$. Thus, for such a $j$, if $P_k$ is any projection of $\gamma_k$ on $S_j$ with balls $\{B_i\}$ we have 
\begin{equation}
\label{P_k inequality}
\begin{split}
\ell_{\tau_\epsilon,P_k,S_j}(\gamma_k) &= \sum_i \tau_\epsilon(B_i) \\
&= \sum_i r(B_i) \epsilon \\
&\geq \frac{\epsilon}{2} d(\gamma(t_{k-1}),\gamma(t_k)).
\end{split}
\end{equation}
Now, let $P$ be any projection of $\gamma$ on $S_j$.  By adding the values $t_0, \dots, t_m$ to $P$ we obtain a partition $P'$ from $P$ and subpartitions $P_k$ of $P'$ consisting of the values between $t_{k-1}$ and $t_k$. It is clear that $P'$ has at most $m$ more intervals than $P$.  As $\tau(v) \leq 2 \epsilon s^{-N}$ for all $v \in S_j$, it follows that
\begin{equation}
\label{P' inequality}
\ell_{\tau_\epsilon,P',S_j}(\gamma) - \ell_{\tau_\epsilon,P,S_j}(\gamma) \leq m 2 \epsilon s^{-N}.
\end{equation}
We note $\ell_{\tau_\epsilon,P',S_j}(\gamma) = \sum_k \ell_{\tau_\epsilon,P_k,S_j}(\gamma_k)$.  Combining this with (\ref{P_k inequality}) and (\ref{P' inequality}) yields
\begin{equation*}
\begin{split}
\ell_{\tau_\epsilon,P,S_j}(\gamma) &\geq \frac{\epsilon}{2} \left( \sum_k  d(\gamma(t_{k-1}),\gamma(t_k)) \right) - m 2 \epsilon s^{-N} \\
&> 2  - m 2 \epsilon s^{-N} \\
&\geq 1.
\end{split}
\end{equation*}
As this holds for all large enough $j$, we conclude that $\tau_\epsilon$ is admissible for each $\gamma \in \Gamma$ and hence for $\Gamma_\infty$.

It remains to show $\norm{\tau_\epsilon}_{p,\infty} \to 0$ as $\epsilon \to 0$.  As our hyperbolic filling has bounded valence (Lemma \ref{bounded valence}), we see the number of vertices with level $n$ is comparable to $s^{nQ}$ up to a fixed multiplicative constant.  Thus,  for $\lambda = 2\epsilon s^{-n} \leq 1$ we have 
\begin{equation*}
\# \{v \in V : \tau(v) > \lambda\} \lesssim s^{nQ} \lesssim  \frac{\epsilon^Q}{\lambda^Q}  = \frac{\epsilon^Q \lambda^{p-Q}}{\lambda^p} \lesssim \frac{\epsilon^p}{\lambda^p} 
\end{equation*}
with implicit constants independent of $n$.  From this the limiting behavior of $\norm{\tau_\epsilon}_{p,\infty}$ follows.
\end{proof}
We remark here that the above result does not hold for $p < Q$.  Indeed, by quasisymmetric invariance (Theorem \ref{QS inv wcqcap}) there are spaces with path families  $\Gamma$ of nonrectifiable curves for which $\wcpcap(\Gamma) > 0$ holds for some $p$.

We also note if $\Gamma_1$ and $\Gamma_2$ are arbitrary path families, then 
\begin{equation*}
\wcqcap(\Gamma_1) \leq \wcqcap(\Gamma_1 \cup \Gamma_2) \leq \wcqcap(\Gamma_1) + \wcqcap(\Gamma_2).
\end{equation*}
This follows as if $\tau$ is admissible for $\Gamma_1 \cup \Gamma_2$, then $\tau$ is admissible for $\Gamma_1$ and if $\tau_1, \tau_2$ are admissible for $\Gamma_1, \Gamma_2$, then $\max\{\tau_1,\tau_2\}$ is admissible for $\Gamma_1 \cup \Gamma_2$.  With this observation and Lemma \ref{wcpcap nonrectifiable} it follows that for any path family $\Gamma$ one has $\wcqcap(\Gamma) = \wcqcap(\Gamma \setminus \Gamma_\infty)$.  Thus, we may assume in the following that all path families $\Gamma$ consist solely of paths with finite length.

%We state some basic properties of $\wcqcap$.  As with $\qmod$, $\wcqcap$ is 0 in the absence of rectifiable curves.  We note $\wcqcap$ also has the monotonicity properties that $\qmod$ has: if $\Gamma_1, \Gamma_2$ are curve families such that every $\gamma \in \Gamma_1$ also satisfies $\gamma \in \Gamma_2$, then $\wcqcap(\Gamma_1) \leq \wcqcap(\Gamma_2)$ and if $\Gamma, \Gamma'$ are curve families such that every curve in $\gamma$ has a subcurve $\gamma'$ in $\Gamma$, then $\wcqcap(\Gamma) \leq \wcqcap(\Gamma')$.

Now we prove Theorems \ref{wcqcap = qmod} and \ref{QS inv wcqcap}.  We start with Theorem \ref{wcqcap = qmod}.  Recall we work in a compact, connected, Ahlfors $Q$-regular metric space $Z$ with hyperbolic filling $X=(V_X, E_X)$.  We also work with a fixed path family $\Gamma$ such that every $\gamma \in \Gamma$ has finite length.
%%%%%%%%%%%%%%%%%%%%%%%%%%
%iffalse here
\iffalse

\begin{lemma}
Let $\gamma \colon [0,L] \to Z$ be a nonconstant path parameterized by arclength.  Then for all $\epsilon > 0$, there is a $\delta > 0$ such that if $I$ is an interval and $B_\delta$ is any ball with radius bounded above by  $\delta$ such that $\gamma(I) \subseteq B_\delta$, then $|I| < \epsilon$.
\end{lemma}

\begin{proof}
Let $\gamma$ be as above.  We note the claim is trivial for $\epsilon > L$, so fix $0 < \epsilon < L$.  Consider the function $g \colon [0, L-\epsilon] \to [0,\infty]$ given by $g(x) = \diam \gamma([x, x+\epsilon])$.  We see that as $\gamma$ is parameterized by arclength, we actually have $g(x) > 0$ for all $x$.  Moreover as $\gamma$ is uniformly continuous it follows that $g$ is continuous.  Hence, $g$ attains a minimum, say $\delta'>0$.  

Setting $\delta = \frac{1}{3}\delta'$, we see that if $B_\delta$ is a ball with radius bounded above by $\delta$, then $\diam B_\delta \leq 2 \delta < \delta'$. Hence, if $\gamma(I) \subseteq B_\delta$, we must have $|I| < \epsilon$ as otherwise we would have a subinterval $I' \subseteq I$ with $|I| = \epsilon$ and $\diam (\gamma(I')) < \diam (\gamma(I)) < 2 \delta < \delta'$, which contradicts the fact that $\delta'$ is the minimum for $g$.
\end{proof}

\fi
%end iffalse

%MAIN THEOREM STARTS HERE

\begin{proof}[Proof of Theorem \ref{wcqcap = qmod}]
We first prove $\wcqcap(\Gamma) \lesssim \qmod(\Gamma)$.   Let $\rho \colon Z \to [0,\infty]$ be an admissible function for the $Q$-modulus of $\Gamma$. 
As $Z$ is compact, we may assume $\rho$ is lower semicontinuous (this follows from the Vitali-Carath\'eodory theorem, see \cite[Section 4.2]{HKST}).  Define $\tau \colon V \to \R$ by 
\begin{equation*}
\tau(v) = r(B_v)\left ( \dashint_{B_v} 10 \rho\right)
\end{equation*}
for $v \in V$.  We show that $\tau$ is admissible for covering capacity.

Fix $\gamma \in \Gamma$.  We recall we assume $\gamma$ has finite length $\ell(\gamma) > 0$.  Set $I = [0, \ell(\gamma)]$; we work with partitions of $I$ and the arclength parameterization of $\gamma$ as in the remark in the introduction.  As $\rho$ is lower semicontinuous, there is a sequence of continuous functions $f_n \geq 0$ such that $f_n$ increases pointwise to $\rho$ (see \cite[Section 4.2]{HKST}).  By the monotone convergence theorem $\int_\gamma 10f_n $ increases to $\int_\gamma 10\rho$ and, as $\int_\gamma 10\rho \geq 10$, there is an $N$ such that for $n \geq N$, we have $\int_\gamma 10f_n \geq 7 $.  Set $f = 10 f_N$ and $M = \max_{z\in Z} f(z)$.

Let $\epsilon = \frac{1}{\ell(\gamma) + 1 }$.  As $Z$ is compact, $f$ is uniformly continuous. Hence there is a $\delta_1 > 0$ such that if $d(x,y) < \delta_1$, then $|f(x)-f(y)| < \epsilon$.   

We find a partition of $[0, \ell(\gamma)]$ given by $0 = x_0 < \dots < x_p = \ell(\gamma)$ with $x_{k+1} - x_k < \delta_1$ and find a $\delta_2 > 0$ such that for each $i$, every $x,y$ in the $\delta_2$ neighborhood of $\gamma_i = \gamma([x_{i-1},x_i])$ satisfies $|f(x)-f(y)| < \epsilon$.  The existence of this partition and of $\delta_2$ follow from the uniform continuity of $f$ on $Z$.   We also set $m_i$ to be the infimum of the values of $f$ on the $\delta_2$ neighborhood of $\gamma_i$.   We further partition each $[x_{i-1},x_i]$ as $x_{i-1} = y^i_0 < \dots < y^i_{q_i} = x_i$ such that $\ell(\gamma_i) - \sum_j d(\gamma(y^i_{j-1}),\gamma(y^i_j)) < \frac{\epsilon}{M p}$.  Set $\delta_3 = \frac{1}{10} \min_{i,j}d(\gamma(y^i_{j-1}),\gamma(y^i_j))$ which we may assume is positive by appropriately choosing $y^i_j$.

%By the lemma, there is a $\delta_4 > 0$ such that if $I$ is partitioned into intervals $I_k$ such that each $\gamma(I_k)$ is completely contained in a ball of radius less than $\delta_4$, then we have $|I_k| < \epsilon$ for each $k$.  

Set $\delta = \min(\frac{1}{3}\delta_1,\delta_2,\delta_3)$.

Let  $\mathscr{S} = \{S_n\}$ be an expanding sequence of covers.  Then, as $\mathscr{S}$ is expanding, for large enough $n$ it follows that $r(B_v) < \delta$ for all $v \in S_n$.  We work with one of these covers $S_n$ with large $n$ which we denote $S$.  Let $P \colon [0,\ell(\gamma)] \to V$ be a projection of $\gamma$ onto $S$ with $t_0, \dots, t_m$ partitioning $[0,\ell(\gamma)]$ and $v_1,\dots, v_m$ vertices such that $\gamma([t_{k-1}, t_k]) \subseteq B_{v_k}$.  % By the lemma we have $t_k - t_{k-1} < \epsilon$ for each $k$ and so

%\begin{equation*}
%\biggl |\int_\gamma f - \sum_k f(\gamma(t_k)) (t_k - t_{k-1})\biggr| < 1.
%\end{equation*}

We have
\begin{equation*}
\int_\gamma f = \sum_i \int_{\gamma_i} f  
\end{equation*}
and we see 
\begin{equation*}
\left|\sum_i \left(\int_{\gamma_i} f - m_i \ell(\gamma_i)\right) \right| \leq \epsilon \sum_i \ell(\gamma_i) < 1
\end{equation*}
and so
\begin{equation*}
\sum_i m_i \ell(\gamma_i) \geq 6.
\end{equation*}

Now we group the $t_k$ into $T_1, \dots, T_p$ where $T_i = \{t_k : x_{i-1} \leq t_k \leq x_i \}$ and similarly write $K_i = \{k : t_k \in T_i\}$.  By our choice of $\delta$ above, we see for each $i$ that
\begin{equation*}
|m_i - f(\gamma(t_k))| < \epsilon
\end{equation*}
whenever $k \in K_i$.  Using $\ell(\gamma_i) - \sum_j d(\gamma(y^i_{j-1}),\gamma(y^i_j)) < \frac{\epsilon}{M p}$, we deduce
\begin{equation*}
\sum_i  m_i \left( \sum_j d(\gamma(y^i_{j-1}),\gamma(y^i_j)) \right) \geq 6 - \epsilon \geq 4.  
\end{equation*}
Now, as $\delta \leq \delta_3$ we may replace $\sum_j d(\gamma(y^i_{j-1}),\gamma(y^i_j))$ in the above sum with twice the sum of the radii of balls $B_v$ from our partition $P$ with corresponding intervals intersecting $\cup_j [y^i_{j-1},y^i_j]$.  That is, we have 
\begin{equation*}
\sum_i\sum_{k \in K_i} m_i 2 r (B_{v_k}) \geq 4.
\end{equation*}
Thus,
\begin{equation*}
\sum_i\sum_{k \in K_i} m_i r (B_{v_k}) \geq 2.
\end{equation*}
We note that for $k \in K_i$ we have
\begin{equation*}
\left( \dashint_{B_{v_k}} f \right) - m_i \geq 0.
\end{equation*}
From this we conclude 
\begin{equation*}
\sum_i\sum_{k \in K_i} \left( \dashint_{B_{v_k}} f \right) r (B_{v_k}) \geq \sum_i\sum_{k \in K_i} m_i r (B_{v_k}) \geq 2.
\end{equation*}
Lastly we deal with the overestimation possible from having $k \in K_i$ for more than one $i$.  This only happens if $t_k = x_i$ for some $i$, which happens at most $p+1$ times (recall our partition is $x_0,\dots,x_p$). We note that $f$ is bounded and that in an expanding sequence of covers we have $r(B_{v_k}) \to 0$ in the above sum.  Thus, if  $f \leq M$ and $r(B_{v_k}) \leq \nu(n)$, double counting such $k$ adds at most $(p+1) M \nu(n)$ to our estimate.  We conclude
\begin{equation*}
\sum_{k=1}^m \tau(v_k) \geq \sum_{k=1}^m \left( \dashint_{B_{v_k}} f \right) r (B_{v_k}) \geq 2 - (p+1) M \nu(n).
\end{equation*}
      From this we see that for large enough $n$ the $\tau$-length of any partition $P$ of $\gamma$ onto $S_n$ is at least 1.  That is, $\tau$ is admissible for $\gamma$ relative to $\mathscr{S}$.  As $\mathscr{S}$ was arbitrary, it follows that $\tau$ is admissible for $\gamma$.

%Now, suppose $\gamma$ has infinite length and write $\gamma \colon I \to Z$ as an arclength parameterization. As in the above, we may approximate the length of a subcurve $\gamma'$ by  $\sum_j d(\gamma(y^i_{j-1}),\gamma(y^i_j))$ to within 1 for some $y_0, \dots, y_q \in I$.  With $\delta$ small relative to this length approximation (as above), if $n$ is large enough such that $S_n$ consists of balls of radius smaller than $\delta$, we conclude $\ell_{\tau, P, S_n}(\gamma) \geq \eta (\ell(\gamma') - 1)$.  As $\gamma$ has infinite length, we can let $\ell(\gamma') \to \infty$ to see $\rho$ is admissible for $\gamma$.

As this holds for all $\gamma \in \Gamma$ we see $\tau$ is admissible for covering capacity.  It remains to show $\norm{\tau}_{Q,\infty}^Q \lesssim \norm{\rho}_Q^Q$ but this follows as in the proof of Theorem \ref{qcap < qmod} with $p=1$.  

%For comparability to $\norm{\rho'}_Q^Q$, we note all that the above required was a fixed $\eta>0$ and so we are free to take $\eta$ as small as necessary to ensure this.

% SECOND THEOREM STARTS HERE

%\begin{thm}
%Let $Q>1$ and let $(Z,d,\mH)$ be a compact, connected Ahlfors Q-regular metric space.  Let $\Gamma$ a closed curve family consisting of rectifiable curves of uniformly bounded length.  Then, $\qmod(\Gamma)  \lesssim \wcqcap(\Gamma)$.
%\end{thm}

We now prove the other direction, namely $\qmod \lesssim \wcqcap $.

Let $\tau \colon V \to \R$ be admissible for covering capacity.  Let 
\begin{equation*}
\sigma_n  = 2 \sum_{v \in V_n} \frac{\tau(v)}{r(B_v)}\chi_{2B_v}
\end{equation*}

As in the proof that $\qmod(A,B) \lesssim \qcap(A,B)$ for open sets, we note that there is a subsequence $\sigma_{n_i}$ with $\norm{\sigma_{n_i}}_Q^Q \lesssim  \norm{\tau}_{Q,\infty}^Q$.  Applying Mazur's Lemma to this subsequence, as in the proof of Theorem \ref{qmod < qcap}, we get convex combinations $\rho_k$ of $\sigma_{n_i}$ with $i \geq k$ and a limit function $\rho$ with $\rho_k \to \rho$ in $L^Q$.  Similarly to that proof, by applying Fuglede's Lemma we may pass to a subsequence and assume that for for all paths $\gamma$ except in a family $\Gamma_0$ of $Q$-modulus $0$ we have $\int_\gamma \rho_n \to \int_\gamma \rho$.  We note that $\norm{\rho}_Q^Q \lesssim \norm{\tau}_{Q,\infty}^Q$.

As the $Q$-modulus of $\Gamma_0$ is 0, there exists a function $\sigma \geq 0$ with $\int_Z \sigma^Q < \infty$ such that for $\gamma \in \Gamma_0$ we have $\int_\gamma \sigma = \infty$.  We claim that $\rho + c\sigma$ is admissible for modulus for any $c>0$.  For $\gamma \in \Gamma_0$, admissibility is clear, so suppose $\gamma \notin \Gamma_0$.  We see if $B_{v_1},\dots,B_{v_M}$ is a sequence of balls which $\gamma$ passes through and that $\ell(\gamma \cap 2B_{v_k}) \geq r(B_{v_k})$ for each $v_k$, then
\begin{equation*}
\begin{split}
\int_\gamma \sigma_n &= 2 \sum_{v \in V_n} \ell(\gamma \cap 2 B_v) \frac{\tau(v)}{r(B_v)} \\
&\geq 2 \sum_{k=1}^M r(B_{v_k}) \frac{\tau(v_k)}{r(B_{v_k})}.
\end{split}
\end{equation*}
Let $S_j = \{v \in V: j \leq \ell(v) \leq 2j \}$ be the set of all vertices with levels between $j$ and $2j$.  As $\tau$ is admissible, it is admissible for the expanding sequence of covers $\{S_j\}$ and hence our integral is bounded below by 1 for large enough $n$.  Thus, $\int_\gamma \rho = \lim_{n\to \infty} \int_\gamma \rho_n \geq 1$.  We conclude that $\rho + c \sigma$ is admissible for modulus and, as $\norm{\rho + c\sigma}_Q^Q \lesssim \norm{\rho}_Q^Q + c\norm{\sigma}_Q^Q$, this shows $\qmod(\Gamma) \lesssim \wcqcap(\Gamma)$.
\end{proof}

Lastly we prove Theorem \ref{QS inv wcqcap}.

\begin{proof}[Proof of Theorem \ref{QS inv wcqcap}]
Recall $Z$ and $W$ are compact, connected, Ahlfors regular metric spaces and $\varphi \colon Z \to W$ is an $\eta$-quasisymmetry.  Let $X = (V_X, E_X)$ and $Y = (V_Y, E_Y)$ by corresponding hyperbolic fillings.
Fix a path family $\Gamma$ in $Z$.  Let $\tau$ be admissible for $\wcpcap(\Gamma)$.  Let $G \colon Y \to X$ denote the quasi-isometry induced by $\varphi^{-1}$ from Lemma \ref{QI induces QS}.  We define $\sigma \colon V_Y \to [0,\infty]$ by $\sigma(y) = \tau(G(y))$.  

We claim $\sigma$ is admissible for $\wcpcap(\varphi(\Gamma))$.  To prove this, let $\mathscr{S}' = \{S_n'\}$ be an expanding sequence of covers in $Y$.  We note that if $\{v_k\}$ is the set of vertices in $S_n'$ then $W \subseteq \bigcup_k B_{v_k}$.  Recall that for a vertex $v \in V_Y$ we defined $G(v)$ such that $\varphi^{-1}(B_v) \subseteq B_{G(v)}$.  Hence,
\begin{equation*}
Z = \varphi^{-1}(W) \subseteq \cup_k \varphi^{-1}(B_{v_k}) \subseteq \cup_k B_{G(v_k)}
\end{equation*}
and we see that $\{S_n\} = \{G(S_n')\}$ is an expanding sequence of covers.

We fix a rectifiable $\gamma' = \varphi(\gamma) \in \varphi(\Gamma)$.  Now, let $P$ be a projection of $\gamma'$ onto $S_n'$, say with balls $B_{y_1},\dots,B_{y_m}$.  From the above, $\varphi^{-1}(B_{y_k}) \subseteq B_{G(y_k)}$ and so the sequence $B_{G(y_k)}$ forms a partition of $\gamma$ using balls in $S_n$. Hence, for large enough $n$, we see $\sum_k \tau(G(y_k)) \geq 1$.  As $\tau(G(y_k)) = \sigma(y_k)$, it follows that $\sigma$ is admissible for $\wcpcap(\varphi(\Gamma))$.

It remains to show $\norm {\sigma}_{p,\infty} \lesssim \norm {\tau}_{p,\infty}$ for which we use Lemma \ref{BoS Lemma}.  To apply Lemma \ref{BoS Lemma} we set $J \subseteq V_Y \times V_X$ where $(y,x) \in J$ if $x = G(y)$.  We see for $J^x = \{y: (y,x) \in J\}$ we have 
\begin{equation*}
|J^x| = |\{y : x = G(y) \}|. 
\end{equation*}
If $x = G(y) = G(y')$ then, as $G$ is a quasi-isometry, it follows that there is a there is a fixed $D'>0$ such that $|y - y'| \leq D'$.  From Lemma \ref{bounded valence} it follows there is a uniformly bounded number of such $y$; that is, $|J^x|$ is uniformly bounded.

Now, $J_y = \{x: (y,x) \in J\} = \{G(y)\}$ so $|J_y| = 1$.
  Lastly, we use $s_y = \sigma(y)$ and $t_x = \tau(x)$ for our sequences.  We have
\begin{equation*}
\sigma(y) = \tau(G(y)) = \sum_{x \in J_y} \tau(x)
\end{equation*}
as $G(y) \in J_y$.  Thus,  $\norm {\sigma}_{p,\infty} \lesssim \norm {\tau}_{p,\infty}$ and so $\wcpcap(\varphi(\Gamma)) \lesssim \wcpcap(\Gamma)$.  The other inequality follows from considering $\varphi^{-1}$ in place of $\varphi$.
\end{proof}

\end{document}